\documentclass{article}
\usepackage[T1]{fontenc}
\usepackage[utf8]{inputenc}
\usepackage{amsmath,amssymb,amsthm}
\usepackage[
a4paper,
left=2.5cm,
right=2.5cm,
top=2.5cm,
bottom=3.0cm
]{geometry}
\usepackage{mathrsfs}
\usepackage{xcolor}
\usepackage{microtype}
\usepackage{tikz}
\usepackage{graphicx}
\usepackage{booktabs}

\definecolor{mycitegreen}{rgb}{0,0.5,0}
\usepackage[
colorlinks=true,
linkcolor=blue,
citecolor=mycitegreen,
urlcolor=blue
]{hyperref}

\usepackage{float}

\usepackage{authblk}

\usepackage{placeins}

\usetikzlibrary{arrows.meta,positioning,fit}

\numberwithin{equation}{section}

\theoremstyle{plain}
\newtheorem{theorem}{Theorem}[section]
\newtheorem{lemma}{Lemma}[section]
\newtheorem{remark}{Remark}[section]
\newtheorem{assumption}{Assumption}[section]

\newcommand{\RR}{\mathbb{R}}

\newcommand{\SSS}{\mathbb{S}}

\newcommand{\supp}{\operatorname{supp}}

\newcommand{\Lip}{\operatorname{Lip}}
\newcommand{\Id}{\operatorname{Id}}

\newcounter{method}
\renewcommand{\themethod}{\Alph{method}}

\newenvironment{method}[1]{%
	\refstepcounter{method}%
	\par\medskip
	\noindent\textbf{Method~\themethod: #1.}\ }%
    {\par\medskip}

\begin{document}
	\title{\textbf{A Unified DeepONet Framework for Logarithmically Stable Infinite-Dimensional Inverse Problems}}
	
	\author[1]{Wen-Jie Wu}
	\author[1,2]{Tiexiang Li\thanks{Corresponding author. E-mail addresses: \texttt{wenjiewu@seu.edu.cn} (W.-J.~Wu), \texttt{txli@seu.edu.cn} (T.~Li), and \texttt{wwlin@outlook.com} (W.-W.~Lin).}}
	\author[2]{Wen-Wei Lin}
	
	\affil[1]{School of Mathematics and Shing-Tung Yau Center, Southeast University, Nanjing 211189, People's Republic of China}
	
	\affil[2]{Shanghai Institute for Mathematics and Interdisciplinary Sciences (SIMIS), Shanghai 200433, People's Republic of China}

	\date{}
	
	\maketitle
	
	\begin{abstract}
		We develop a unified DeepONet framework for logarithmically stable inverse problems between infinite-dimensional function spaces, with inverse acoustic scattering as a model application. The framework is formulated at the operator level by separating the learned inverse map into measurement encoding, finite-dimensional neural approximation, and functional reconstruction components. For inverse maps satisfying a logarithmic stability estimate, we establish quantitative a priori error bounds that separate the encoder, finite-dimensional neural approximation, and reconstruction contributions. For prescribed encoder and reconstruction ranks, we obtain a network-size-dependent bound for the finite-dimensional neural approximation error, together with rank-dependent bounds for the encoding and reconstruction errors. For comparison, we also record the corresponding Lipschitz-stable estimate arising from the same error decomposition. The quantitative inverse-scattering analysis is then specialized, in three dimensions, to the recovery of a medium contrast from fixed-frequency far-field measurements. Numerical experiments in two and three dimensions, using both function-space and finite-dimensional priors, illustrate the reconstruction performance and empirical sensitivity to synthetic measurement noise.
	\end{abstract}
	
	
	

	\section{Introduction}
	Deep operator networks (DeepONets) are neural-network architectures designed to approximate nonlinear operators between function spaces. Their theoretical origin is closely related to the universal approximation theory for continuous functionals and nonlinear operators developed by Chen and Chen \cite{CC93,CC95}. Motivated by this operator approximation viewpoint, Lu, Jin and Karniadakis introduced the DeepONet architecture, which was further developed and systematically demonstrated in subsequent work \cite{LJK19,LJP+21}. In the classical DeepONet architecture, the branch network maps finitely many measurements of the input function to a coefficient vector, whereas the trunk network maps an output query point to a feature vector; their inner product yields the predicted value of the output function at that point. The framework developed below generalizes this construction by allowing abstract measurement encoders and general finite-dimensional reconstructors. Lu et al.~demonstrated the effectiveness of DeepONet for explicit operators, including integrals and fractional Laplacians, as well as implicit solution operators associated with deterministic and stochastic differential equations \cite{LJP+21}. DeepONets, Fourier neural operators, and related architectures are widely studied methods for operator learning, which seeks to approximate nonlinear operators between infinite-dimensional function spaces \cite{KLL+23}.
	
	The approximation properties of DeepONets have been studied in several recent works. Lanthaler, Mishra and Karniadakis established error estimates for DeepONets as a learning framework in infinite dimensions \cite{LMK22}. Deng, Shin, Lu, Zhang and Karniadakis analyzed approximation rates of DeepONets for solution operators arising from advection--diffusion equations \cite{DSL+22}. Marcati and Schwab proved exponential convergence rates of deep operator networks for coefficient-to-solution maps of analytic elliptic PDEs \cite{MS23}. These results support the use of DeepONet as a genuine operator approximation method for maps defined on infinite-dimensional input spaces.
	
	Machine-learning methods for inverse problems include data-driven reconstruction, learned iterative algorithms, and physics-informed formulations; see \cite{AMOS19,AO17,AO18,RPK19}. Since the data in many PDE inverse problems are naturally represented by measurement functions or boundary input--output operators, while the unknowns are coefficient functions, operator learning provides a natural framework. De Hoop, Lassas and Wong studied deep learning architectures for nonlinear operator functions and nonlinear inverse problems \cite{HLW21}. Molinaro, Yang, Engquist and Mishra introduced Neural Inverse Operators (NIOs), based on a suitable composition of DeepONets and Fourier neural operators, for approximating operator-to-function inverse maps arising in PDE inverse problems \cite{MYEM23}. In the context of Calder\'on's problem, Castro, Mu\~noz and Valenzuela proved that the direct and inverse Calder\'on mappings can be rigorously approximated by DeepONets \cite{CMV24}. For electrical impedance tomography, Abhishek and Strauss formulated the recovery of a conductivity from the Neumann-to-Dirichlet operator as an operator-to-function learning problem and studied its approximation by a DeepONet-type architecture \cite{AS26}. For full-waveform inversion, Zhu et al. proposed Fourier-DeepONet, which incorporates source parameters as additional inputs and reconstructs subsurface structures from seismic waveform data across varying source frequencies and locations \cite{ZFLL23}. These works illustrate the potential of DeepONet and related
	neural-operator methods for inverse problems with function- or
	operator-valued data and function-valued unknowns.
	
	Logarithmic stability is a classical feature of many severely ill-posed infinite-dimensional inverse problems. Foundational examples include inverse boundary value problems such as the Calder\'on and Gel'fand--Calder\'on problems, where logarithmic stability estimates and exponential instability results indicate that the inverse map generally has at best a logarithmic modulus of continuity \cite{Ale88,Man01,Nov11}. Logarithmic stability is also central in inverse scattering. Stefanov proved stability estimates of logarithmic type for inverse potential scattering at fixed energy \cite{Ste90}. Isakov obtained logarithmic stability estimates for obstacle inverse scattering from scattering amplitudes \cite{Isa92}. H\"ahner and Hohage established logarithmic stability estimates for inverse acoustic inhomogeneous medium scattering from far-field and near-field data \cite{HH01}. Related stability and regularization results for inverse scattering and exponentially ill-posed inverse medium problems were further developed in \cite{Sin06,HW15}. For broader treatments of ill-posedness, stability, and regularization in inverse PDE and inverse scattering problems, we refer to the monographs \cite{Isa06,CK19}.
	
	By contrast, Lipschitz-type stability is typically available only under additional structural restrictions on the unknown. Examples include piecewise constant conductivities with finitely many unknown values \cite{AV05}, abstract inverse problems including the inverse medium problem for the Helmholtz equation \cite{Bou13}, and finite-measurement inverse problems with finite-dimensional a priori structure \cite{AS22}. In a related neural-network analysis, Lerma Pineda and Petersen \cite{LP23} obtained stable approximation results by restricting an infinite-dimensional ill-posed forward map to
	finite-dimensional subspaces on which the inverse is Lipschitz
	continuous. These results rely on finite-dimensional or structurally restricted regimes. The present work instead propagates the logarithmic stability modulus of the physical inverse map through the DeepONet error decomposition without replacing it by a finite-dimensional Lipschitz inverse.
	
	\subsection{Main contributions}
	
	The existing results described above establish universal approximation properties, problem-specific neural inversion methods, or stable neural-network approximations under finite-dimensional Lipschitz restrictions. To the best of our knowledge, these works do not provide an operator-level estimate that tracks an infinite-dimensional logarithmic stability modulus through measurement encoding, finite-dimensional neural approximation, and output reconstruction.
	
	The main contributions of this work are as follows. Building on the DeepONet error decomposition in \cite{LMK22}, we derive a unified error estimate that separates the encoding, finite-dimensional neural approximation, and reconstruction contributions. Moreover, for fixed encoder and reconstruction dimensions, we establish a network-size-dependent ReLU approximation bound when the physical inverse map satisfies logarithmic stability, and we give the corresponding Lipschitz-stable comparison. Finally, we specialize the abstract framework to three-dimensional fixed-frequency inverse acoustic medium scattering and derive rank-dependent estimates for spectral encoding and polynomial reconstruction, leading to a quantitative DeepONet error bound. Numerical experiments in two and three dimensions illustrate the reconstruction performance for both infinite- and finite-dimensional prior classes.
	
	\subsection{Mathematical setting and DeepONet architecture}
	
	We next introduce the abstract setting and the
	encoder--approximator--reconstructor factorization used in the analysis below.
	Let $X$ and $Y$ be separable Hilbert spaces representing the data and
	reconstruction spaces, respectively, and let $D_G\subset X$ be a closed
	admissible data set. We consider a continuous physical inverse map
	\begin{equation*}
		G:D_G\to Y.
	\end{equation*}
	Our goal is to approximate $G$ on the admissible data by a DeepONet and to
	quantify the contributions of encoding, neural approximation, and
	reconstruction to the total approximation error.
	
	To formulate the target operator on all of $X$, we use the following
	extension theorem.
	
	\begin{lemma}[\cite{Dug51}]\label{lem:dugundji}
		Let $X$ be a metric space, let $A\subset X$ be closed, and let $Y$ be a
		locally convex linear space. Then every continuous map
		\begin{equation*}
			f:A\to Y
		\end{equation*}
		admits a continuous extension
		\begin{equation*}
			F:X\to Y
		\end{equation*}
		such that $F|_A=f$.
	\end{lemma}
	
	Since $Y$ is a Hilbert space and hence locally convex,
	Lemma~\ref{lem:dugundji} yields a continuous extension
	\begin{equation*}
		\mathscr G:X\to Y,
		\qquad
		\mathscr G|_{D_G}=G,
	\end{equation*}
	which serves as the target operator for the DeepONet.
	
	To begin, we define an \textit{encoder}, $\mathcal{E}:X\to\mathbb{R}^{m}$. Next, we define an \textit{approximator} $\mathcal{A}:\mathbb{R}^{m}\to\mathbb{R}^{p}$ for some $p\in\mathbb{N}$, which is a feedforward deep neural network. Precisely, a feedforward deep neural network is described by
	\begin{equation}\label{eq:neural-network}
		\Phi_{\theta}(z)=C_{L}\circ\sigma\circ C_{L-1}\circ\cdots\circ\sigma\circ C_{2}\circ\sigma\circ C_{1}(z),
	\end{equation}
	where $z\in\mathbb{R}^{d_{\mathrm{in}}}$ denotes the input vector, and $\sigma$ is the ReLU activation function, $\sigma(x)=\max\{x,0\}$. For $1\leq\ell\leq L$, $C_{\ell}$ is defined by
	\begin{equation*}
		C_\ell(z_\ell):=W_\ell z_\ell+b_\ell,\quad\text{where}\quad W_{\ell}\in\mathbb{R}^{d_{\ell+1}\times d_{\ell}},\;\;b_{\ell}\in\mathbb{R}^{d_{\ell+1}}.
	\end{equation*}
	Here we also denote $d_{\mathrm{in}}=d_{1}$ and $d_{\mathrm{out}}=d_{L+1}$. The set $\theta=\{\theta_{\ell}\}_{\ell=1}^{L}=\{W_{\ell},b_{\ell}\}_{\ell=1}^{L}\in \Theta\subset\mathbb{R}^{M}$ consists of the learnable parameters, where $M=\sum_{\ell=1}^{L}(d_{\ell+1}d_{\ell}+d_{\ell+1})$. For a matrix or vector $A$, let $\|A\|_0$ denote the number of its
	nonzero scalar entries. We define the size and depth of
	$\Phi_{\theta}$ by
	\begin{equation*}
		\mathrm{size}(\Phi_{\theta})
		:=
		\sum_{\ell=1}^{L}
		\left(
		\|W_{\ell}\|_0+\|b_{\ell}\|_0
		\right),
		\qquad
		\mathrm{depth}(\Phi_{\theta})
		:=
		L-1.
	\end{equation*}
	
	As in the DeepONet architecture, we define the \textit{branch net} as the composition of $\mathcal{A}$ and $\mathcal{E}$, i.e., $\boldsymbol{\beta}=\mathcal{A}\circ\mathcal{E}:X\to\mathbb{R}^{p}$, where $\mathcal{A}$ is given as in \eqref{eq:neural-network}. A \textit{reconstructor} $\mathcal{R}:\mathbb{R}^{p}\to Y$ is defined by
	\begin{equation*}
		\mathcal R(c):=\tau_0+\sum_{k=1}^p c_k\tau_k,
		\qquad c=(c_1,\ldots,c_p)\in\mathbb R^p,
	\end{equation*}
	where $\tau_0,\tau_1,\ldots,\tau_p\in Y$, and $\boldsymbol{\tau}=(\tau_{0},\tau_{1},\ldots,\tau_{p})$ is called a \textit{trunk net}. In the classical neural-trunk realization of DeepONet, each $\tau_k$ may be represented by a feedforward neural network. Other choices of $\mathcal R$ are discussed in Subsection~\ref{subsec:reconstruction}. Putting all components together, we now have a DeepONet:
	
	\begin{equation*}
		\mathscr{N}:X\to Y,\quad\mathscr{N}(f)=(\mathcal{R}\circ\mathcal{A}\circ\mathcal{E})(f).
	\end{equation*}
	
	Here $\mathcal E$, $\mathcal A$, and $\mathcal R$ encode the data, predict reconstruction coefficients, and reconstruct the output, respectively; for $f\in D_G$, $\mathscr N(f)$ is intended to approximate $G(f)=\mathscr G(f)$.
	
	We assume that $\mathcal E:X\to\mathbb R^m$ is Borel measurable. Since $\mathcal A$ is a finite composition of continuous affine maps and the ReLU activation, and $\mathcal R$ is an affine map from $\mathbb R^p$ to $Y$, both $\mathcal A$ and $\mathcal R$ are continuous. Consequently, $\mathscr N=\mathcal R\circ\mathcal A\circ\mathcal E$ is Borel measurable.
	
	To describe the approximation error on the physical data, let
	$\mu$ be a Borel probability measure on $X$ satisfying
	$\mu(D_G)=1$. We assume that $\mu$ has a finite second moment, i.e., $\int_X\|f\|_X^2\,d\mu(f)<\infty$.
	Since $\mathscr G:X\to Y$ is continuous, it is Borel measurable.
	We further assume that $\mathscr G\in L^2(X,\mu;Y)$, or equivalently,
	that $\mathscr G_\#\mu$ has finite second moment in $Y$. We define
	\begin{equation}
		\mathscr{E}:=
		\left(
		\int_X
		\|\mathscr{G}(f)-\mathscr{N}(f)\|_Y^2\,d\mu(f)
		\right)^{1/2}.
	\end{equation}
	
	We now focus on the quantitative error mechanism arising from the three components of the DeepONet: the encoder, the finite-dimensional neural approximator, and the reconstructor. Following the strategy used in \cite{LMK22}, let $\mathcal D:\mathbb R^m\to X$ and $\mathcal P:Y\to\mathbb R^p$ be Borel measurable maps, called the \textit{decoder} and \textit{projector}, respectively, satisfying
	\begin{equation*}
		\mathcal D\circ\mathcal E\approx\Id:X\to X,
		\qquad
		\mathcal R\circ\mathcal P\approx\Id:Y\to Y.
	\end{equation*}
	The maps \(\mathcal D\) and \(\mathcal P\) are not necessarily unique and will be chosen according to the inverse problem under consideration. We have the following diagram for all mappings:
	\begin{center}
		\begin{tikzpicture}[>=stealth]
			\node (X)  at (0,2) {$X$};
			\node (Y)  at (2.4,2) {$Y$};
			\node (Rm) at (0,0) {$\mathbb{R}^m$};
			\node (Rp) at (2.4,0) {$\mathbb{R}^p$};
			
			\draw[->] (X) -- node[above] {$\mathscr{G}$} (Y);
			\draw[->] (Rm) -- node[below] {$\mathcal{A}$} (Rp);
			
			\draw[->, dashed]
			([xshift=-0.35em]X.south) -- node[left] {$\mathcal{E}$}
			([xshift=-0.35em]Rm.north);
			\draw[->]
			([xshift= 0.35em]Rm.north) -- node[right] {$\mathcal{D}$}
			([xshift= 0.35em]X.south);
			
			\draw[->]
			([xshift=-0.35em]Y.south) -- node[left] {$\mathcal{P}$}
			([xshift=-0.35em]Rp.north);
			\draw[->, dashed]
			([xshift= 0.35em]Rp.north) -- node[right] {$\mathcal{R}$}
			([xshift= 0.35em]Y.south);
		\end{tikzpicture}
	\end{center}
	With the choices of \(\mathcal D\) and \(\mathcal P\), we define
	\begin{equation}
		\mathscr{E}_{\mathcal{E}}:=\left(\int_{X}\|\mathcal{D} \circ \mathcal{E}(f)-f\| _{X}^{2} d \mu(f)\right)^{1 / 2},
	\end{equation}
	\begin{equation}
		\mathscr{E}_{\mathcal{A}}:=\left(\int_{\mathbb{R}^{m}}\|\mathcal{A}(x)-\mathcal{P} \circ \mathscr{G} \circ \mathcal{D}(x)\| _{\ell^{2}(\mathbb{R}^{p})}^{2} d\left(\mathcal{E}_{\#} \mu\right)(x)\right)^{1 / 2},
	\end{equation}
	\begin{equation}
		\mathscr{E}_{\mathcal{R}}:=\left(\int_{Y}\|(\mathcal{R} \circ \mathcal{P})(v)-v\| _{Y}^{2} d\left(\mathscr{G}_{\#} \mu\right)(v)\right)^{1 / 2}.
	\end{equation}
	
	The remainder of the paper is organized as follows. Section~\ref{section2} develops an abstract DeepONet framework for logarithmically stable inverse problems. Section~\ref{section3} establishes the corresponding unified error estimate. Section~\ref{section4} specializes the framework to inverse acoustic scattering, derives a quantitative error bound in three dimensions, and presents numerical experiments in two and three dimensions. Section~\ref{section5} concludes the paper.
	
	\section{Abstract DeepONet framework}\label{section2}
	Building on the abstract setup above, this section develops the encoding, neural approximation, and reconstruction components for inverse problems with logarithmic stability.
	
	\begin{assumption}[Logarithmic stability]\label{ass:stability}
		Assume that the inverse map
		\[
		G:D_G\to Y
		\]
		satisfies the following logarithmic stability estimate: there exist constants
		$C_{\mathrm{stab}}>0$, $\beta\in(0,1)$, and $\tau\in(0,1)$ such that
		\[
		\|G(x_1)-G(x_2)\|_Y
		\le
		C_{\mathrm{stab}}\,|\log \|x_1-x_2\|_X|^{-\beta},
		\]
		for all $x_1,x_2\in D_G$ satisfying $\|x_1-x_2\|_X\le \tau$. Here and below, we use the continuous extension convention $|\log 0|^{-\alpha}:=0$ for every $\alpha>0$.
	\end{assumption}
	
	\subsection{Encoding of measurement data}
	
	Let $X_m\subset X$ be an $m$-dimensional subspace, and let $\Pi_m^X:X\to X_m$ be a rank-$m$ approximation operator. Fix a basis $\{\psi_j^{(m)}\}_{j=1}^m$ of $X_m$. We define
	\begin{equation*}
		\mathcal E_m:X\to\RR^m,
		\qquad
		\mathcal E_m(x):=\bigl(c_1(x),\dots,c_m(x)\bigr),
	\end{equation*}
	where
	\begin{equation*}
		\Pi_m^X x=\sum_{j=1}^m c_j(x)\psi_j^{(m)},
	\end{equation*}
	and define
	\begin{equation*}
		\mathcal D_m:\RR^m\to X,
		\qquad
		\mathcal D_m(c):=\sum_{j=1}^m c_j\psi_j^{(m)}.
	\end{equation*}
	We assume throughout that $\mathcal E_m:X\to\RR^m$ is Borel measurable, so that the pushforward measure $(\mathcal E_m)_\#\mu$ is well defined on $\RR^m$.
	Then, by construction,
	\begin{equation*}
		\mathcal D_m\circ \mathcal E_m=\Pi_m^X.
	\end{equation*}
	
	\begin{theorem}[Abstract encoder principle]
		\label{thm:encoder-principle}
		For the encoder--decoder pair induced by $\Pi_m^X$, one has
		\begin{equation*}
			\mathscr E_{\mathcal E}
			=
			\left(
			\int_X \|x-\Pi_m^X x\|_X^2\,d\mu(x)
			\right)^{1/2}.
		\end{equation*}
		In particular, if
		\begin{equation*}
			\int_X \|x-\Pi_m^X x\|_X^2\,d\mu(x)\le \varepsilon_E(m)^2,
		\end{equation*}
		then
		\begin{equation*}
			\mathscr E_{\mathcal E}\le \varepsilon_E(m).
		\end{equation*}
	\end{theorem}
	
	\begin{proof}
		Since $\mathcal D_m\circ \mathcal E_m=\Pi_m^X$,
		\begin{equation*}
			\mathscr E_{\mathcal E}
			=
			\left(
			\int_X \|\mathcal D_m\circ \mathcal E_m(x)-x\|_X^2\,d\mu(x)
			\right)^{1/2}
			=
			\left(
			\int_X \|\Pi_m^X x-x\|_X^2\,d\mu(x)
			\right)^{1/2}.
		\end{equation*}
	\end{proof}
	
	We next describe several concrete encoder constructions covered by this abstract principle.
	\begin{method}{Problem-adapted spectral encoding}
		Assume that $X$ admits a distinguished orthonormal basis $\{\psi_j\}_{j\ge1}$. Define
		\begin{equation*}
			\Pi_m^X x:=\sum_{j=1}^m \langle x,\psi_j\rangle_X\,\psi_j,
		\end{equation*}
		\begin{equation*}
			\mathcal E_m(x):=\bigl(\langle x,\psi_1\rangle_X,\dots,\langle x,\psi_m\rangle_X\bigr),
			\qquad
			\mathcal D_m(c):=\sum_{j=1}^m c_j\psi_j.
		\end{equation*}
		Then
		\begin{equation*}
			\mathcal D_m\circ \mathcal E_m=\Pi_m^X.
		\end{equation*}
		This includes Fourier, spherical-harmonic, and other
		geometry-adapted spectral coefficient encoders. 
	\end{method}

	\begin{method}{Covariance-spectral encoding}
		Assume that $\mu$ has finite second moment, and define its mean by
		\begin{equation*}
			\bar x_\mu:=\int_X x\,d\mu(x).
		\end{equation*}
		Let $\Gamma_\mu$ be the covariance operator of $\mu$, with eigenpairs $(\lambda_j,e_j)$. Define
		\begin{equation*}
			\Pi_m^X x
			=
			\sum_{j=1}^m \langle x-\bar x_\mu,e_j\rangle_X\,e_j+\bar x_\mu,
		\end{equation*}
		\begin{equation*}
			\mathcal E_m(x):=\bigl(\langle x-\bar x_\mu,e_1\rangle_X,\dots,\langle x-\bar x_\mu,e_m\rangle_X\bigr),
			\qquad
			\mathcal D_m(c):=\sum_{j=1}^m c_j e_j+\bar x_\mu.
		\end{equation*}
		Then
		$
			\mathcal D_m\circ \mathcal E_m=\Pi_m^X,
		$
		and
		\begin{equation*}
			\int_X \|x-\Pi_m^X x\|_X^2\,d\mu(x)=\sum_{j>m}\lambda_j.
		\end{equation*}
		Hence
		\begin{equation*}
			\mathscr E_{\mathcal E}^2=\sum_{j>m}\lambda_j.
		\end{equation*}
		This is the covariance-spectral, or Karhunen--Lo\`eve/PCA, encoder underlying the optimal linear encoding principle in \cite{LMK22}. This is the affine counterpart of Theorem~\ref{thm:encoder-principle}, obtained by applying the preceding construction to the centered variable $x-\bar x_\mu$.
	\end{method}

	\begin{method}{Sensor-based encoding}
		Let $\ell_1,\dots,\ell_m\in X^\ast$ be bounded linear functionals, and define
		\begin{equation*}
			\mathcal E_m(x):=\bigl(\ell_1(x),\dots,\ell_m(x)\bigr).
		\end{equation*}
		Given a Borel measurable reference finite-dimensional approximation map $\Pi_m^X:X\to X$, a Borel measurable decoder $\mathcal D_m:\mathbb R^m\to X$ is chosen so that
		$\mathcal D_m\circ\mathcal E_m$ approximates $\Pi_m^X$ on the relevant data distribution. Thus,
		\begin{equation*}
			\mathscr E_{\mathcal E}
			\le
			\left(
			\int_X \|x-\Pi_m^X x\|_X^2\,d\mu(x)
			\right)^{1/2}
			+
			\left(
			\int_X \|\Pi_m^X x-\mathcal D_m\circ \mathcal E_m(x)\|_X^2\,d\mu(x)
			\right)^{1/2}.
		\end{equation*}
		Boundary probes and electrode measurements fall into this class whenever they define bounded linear functionals on $X$. For point-evaluation encoders, the encoder may instead be understood through a measurable extension when point evaluations are not well defined on $X$.
	\end{method}
	
	\subsection{Approximation of the inverse map}
	We next derive ReLU approximation error estimates for the inverse map under logarithmic stability and compare them with the corresponding Lipschitz-stable case.
	\begin{theorem}[Approximation error under logarithmic stability]\label{thm:abstract-approximation}
		Fix $m,p\in\mathbb N$. Suppose Assumption~\ref{ass:stability} holds. Let $\mathcal D:\mathbb R^m\to X$ and $\mathcal P:Y\to\mathbb R^p$ be Lipschitz continuous maps, and let $K:=\supp(\mathcal E_\#\mu)\subset \mathbb R^m$ be compact. Assume that $\mathcal D(K)\subset D_G$, and define $H:=\mathcal P\circ G\circ\mathcal D:K\to \mathbb R^p$.
		Then for every $\eta\in(0,1)$ there exists a ReLU neural network $\mathcal A_\eta:\mathbb R^m\to \mathbb R^p$
		such that
		\begin{equation*}
			\mathscr{E}_{\mathcal{A}}(\mathcal A_\eta)
			:=\|\mathcal A_\eta - H\|_{L^{2}(K,\,\mathcal{E}_{\#}\mu)}
			\le\eta,
		\end{equation*}
		and
		\begin{equation*}
			\mathrm{size}(\mathcal A_\eta)
			\le c_0(p+1)\exp \bigl(c_1m\eta^{-1/\beta}\bigr),
			\qquad
			\mathrm{depth}(\mathcal A_\eta)
			\le c_2+c_3m\eta^{-1/\beta}.
		\end{equation*}
		The constants $c_0,c_1,c_2,c_3>0$ may depend on the prescribed ranks $m,p$ and on the fixed data in the assumptions, but are independent of $\eta$ and $N$.
		For $N\in\mathbb N$, define
		\begin{equation*}
			\mathscr E_{\mathcal A}(N)
			:=
			\inf_{\substack{\mathcal A:\mathbb R^m\to\mathbb R^p\ \text{is a ReLU network}\\
			\mathrm{size}(\mathcal A)\le N}}
			\|\mathcal A-H\|_{L^2(K,\mathcal E_\#\mu)}.
		\end{equation*}
		Set
		\begin{equation*}
			N_0:=c_0(p+1)\exp(c_1m).
		\end{equation*}
		Then, for every $N> N_0$,
		\begin{equation*}
			\mathscr E_{\mathcal A}(N)
			\le
			\left[
			\frac{1}{c_1m}
			\log\!\left(\frac{N}{c_0(p+1)}\right)
			\right]^{-\beta}.
		\end{equation*}
	\end{theorem}
	
	\begin{proof}
		The logarithmic stability of $G$ implies that for all $f_1,f_2\in D_G$ with $\|f_1-f_2\|_X\le \tau$,
		\begin{equation*}
			\|G(f_1)-G(f_2)\|_Y
			\;\le\;
			C_{\mathrm{stab}}\,
			\omega\!\left(
			\|f_1-f_2\|_X
			\right),
			\qquad 
			\omega(t)=|\log t|^{-\beta}.
		\end{equation*}
		Let \(L_{\mathcal D}\) and \(L_{\mathcal P}\) denote the Lipschitz constants of \(\mathcal D\) and \(\mathcal P\), respectively.
		\begin{equation*}
			\|\mathcal D(x_1)-\mathcal D(x_2)\|_X
			\;\le\;
			L_\mathcal D\,\|x_1-x_2\|_{\mathbb R^{m}},
		\end{equation*}
		and therefore,
		\begin{equation*}
			\|H(x_1)-H(x_2)\|_{\mathbb R^{p}}
			\;\le\;
			L_\mathcal P C_{\mathrm{stab}}\,
			\omega\!\left(L_\mathcal D\,\|x_1-x_2\|_{\mathbb R^{m}}\right).
		\end{equation*}
		Setting 
		\begin{equation*}
			C_H = L_\mathcal P C_{\mathrm{stab}}, \qquad c_H = L_\mathcal D,
		\end{equation*}
		we rewrite this as the log-type modulus of continuity
		\begin{equation}\label{eq:log-modulus-final}
			\|H(x_1)-H(x_2)\|_{\mathbb R^{p}}
			\;\le\;
			C_H\,
			\Bigl[-\log\!\bigl(c_H\|x_1-x_2\|_{\mathbb R^{m}}\bigr)\Bigr]^{-\beta},
			\; x_1,x_2\in K ,
			c_H\|x_1-x_2\|_{\mathbb R^m}\le \tau.
		\end{equation}
		Because $K$ is compact, it is contained in a rectangle $[a,b]^m$ of diameter $R$.
		Choose a shifted uniform partition of $[a,b]^m$ into $n$ uniform intervals per coordinate such that every cell boundary has $(\mathcal E_\#\mu)$-measure zero. Let the mesh width
		$h\sim R/n$, producing $J=n^{m}$ cells $\{Q_j\}_{j=1}^{J}$, where $J$ denotes the total number of partition cells.  
		Let $\mathcal J_K:=\{j:Q_j\cap K\neq\emptyset\}$. Then
		$|\mathcal J_K|\le J$. For each $j\in\mathcal J_K$, choose
		$x_j\in Q_j\cap K$. Using a fixed tie-breaking rule on shared boundaries, define the piecewise-constant function
		\begin{equation*}
			H_J(x) := H(x_j), \qquad x\in Q_j\cap K.
		\end{equation*}
		Whenever $x\in Q_j\cap K$, $\|x-x_j\|_{\mathbb R^{m}}\le \sqrt{m} h$. For $h$ small enough so that $c_H \sqrt{m} h\le \min\{\tau,e^{-1}\}$, we may apply the logarithmic stability estimate and obtain by 
		\eqref{eq:log-modulus-final},
		\begin{equation*}
			\|H(x)-H_J(x)\|_{\mathbb R^{p}}
			\;\le\;
			C_H\,\Bigl[-\log\!\bigl(c_H \sqrt{m} h\bigr)\Bigr]^{-\beta}.
		\end{equation*}
		To guarantee
		\begin{equation*}
			\|H - H_J\|_{L^\infty(K)} \le \eta/2,
		\end{equation*}
		it suffices to choose $h$ so that
		\begin{equation*}
			C_H\,\Bigl[-\log\!\bigl(c_H \sqrt{m} h\bigr)\Bigr]^{-\beta} \le \eta/2,
		\end{equation*}
		which is achieved by choosing
		\begin{equation*}
			h \;\asymp\; \exp \bigl(-c_h\,\eta^{-1/\beta}\bigr),
			\qquad
			J=n^{m}\;\lesssim\;\exp \bigl(C\,m\,\eta^{-1/\beta}\bigr).
		\end{equation*}
		
		Next we approximate $H_J$ by a ReLU network.  
		By Lemma~\ref{lem:PC-J} below, applied with tolerance $\eta/2$, there exists a ReLU
		network $\mathcal A_J:\mathbb{R}^{m}\to\mathbb{R}^{p}$ such that
		\begin{equation*}
			\|\mathcal A_J - H_J\|_{L^{2}\!\bigl(K,\,\mathcal{E}_{\#}\mu\bigr)} \le \eta/2,
		\end{equation*}
		and
		\begin{equation*}
			\mathrm{size}(\mathcal A_J)
			\le
			C_*J\left(1+\log\frac{J}{\eta}\right)+pJ,
			\qquad
			\mathrm{depth}(\mathcal A_J)
			\le
			C_*\left(1+\log\frac{J}{\eta}\right).
		\end{equation*}
		With this choice of $h$, after enlarging the constants we have
		\begin{equation*}
			J\le \exp \bigl(Cm\eta^{-1/\beta}\bigr)
		\end{equation*}
		and hence
		\begin{equation*}
			\log\frac{J}{\eta}
			\le C'\left(1+m\eta^{-1/\beta}\right).
		\end{equation*}
		Therefore the contribution of the $J$ parallel indicator subnetworks and the final $p$-dimensional affine output layer satisfies
		\begin{equation*}
			\mathrm{size}(\mathcal A_J)
			\le c_0(p+1)\exp \bigl(c_1m\eta^{-1/\beta}\bigr),
			\qquad
			\mathrm{depth}(\mathcal A_J)
			\le c_2+c_3m\eta^{-1/\beta}.
		\end{equation*}
		Set $\mathcal A_{\eta}:=\mathcal A_J$. Then, by the triangle
		inequality,
		\begin{equation*}
			\|\mathcal A_{\eta} - H\|_{L^{2}\!\bigl(K,\,\mathcal{E}_{\#}\mu\bigr)}
			\;\le\;
			\|\mathcal A_{\eta} - H_J\|_{L^{2}\!\bigl(K,\,\mathcal{E}_{\#}\mu\bigr)}
			+
			\|H_J - H\|_{L^{2}\!\bigl(K,\,\mathcal{E}_{\#}\mu\bigr)}
			\;\le\;
			\eta/2 + \|H_J - H\|_{L^{\infty}(K)}
			\;\le\;
			\eta.
		\end{equation*}
		Therefore,
		\begin{equation*}
			\mathscr{E}_{\mathcal{A}}(\mathcal A_\eta)
			=
			\|\mathcal A_\eta - H\|_{L^{2}\!\bigl(K,\,\mathcal{E}_{\#}\mu\bigr)}
			\;\le\;
			\eta.
		\end{equation*}
		
		For $N> N_0$, define
		\begin{equation*}
			\eta_N
			:=
			\left[
			\frac{1}{c_1m}
			\log\!\left(\frac{N}{c_0(p+1)}\right)
			\right]^{-\beta}.
		\end{equation*}
		Then $\eta_N\in(0,1)$ and
		\begin{equation*}
			\mathrm{size}(\mathcal A_{\eta_N})
			\le
			c_0(p+1)\exp \bigl(c_1m\eta_N^{-1/\beta}\bigr)
			=N.
		\end{equation*}
		Consequently,
		\begin{equation*}
			\mathscr E_{\mathcal A}(N)
			\le
			\|\mathcal A_{\eta_N}-H\|_{L^2(K,\mathcal E_\#\mu)}
			\le \eta_N,
		\end{equation*}
		which proves the stated budget-dependent estimate and completes the proof.
	\end{proof}
	
	The preceding proof relies on the following auxiliary ReLU approximation result for piecewise-constant functions.
	\begin{lemma}[ReLU approximation of piecewise-constant functions]
		\label{lem:PC-J}
		Let $K\subset\mathbb{R}^m$ be compact and let $\nu$ be a finite Borel measure on $K$.
		Let $\{Q_j\}_{j=1}^{J}$ be axis-aligned cubes of side length $h$ with pairwise disjoint interiors whose union contains $K$, and for each $j$ choose $c_j\in\mathbb{R}^p$.
		Using a fixed tie-breaking rule on shared boundaries, define the piecewise-constant function
		\begin{equation*}
			H_J(x)=c_j,\qquad x\in Q_j\cap K .
		\end{equation*}
		Assume $M:=\max_j |c_j|<\infty$ and that
		$\nu(\partial Q_j\cap K)=0$, $j=1,\dots,J$.
		Then for every $\varepsilon\in(0,1)$ there exists a ReLU neural network $\mathcal A_J:\mathbb{R}^m\to\mathbb{R}^p$ such that
		\begin{equation*}
			\|\mathcal A_J-H_J\|_{L^2(K,\nu)}\le \varepsilon,
		\end{equation*}
		and
		\begin{equation*}
			\mathrm{size}(\mathcal A_J)
			\le C_*J\left(1+\log\frac{J}{\varepsilon}\right)+pJ,
			\qquad
			\mathrm{depth}(\mathcal A_J)
			\le C_*\left(1+\log\frac{J}{\varepsilon}\right),
		\end{equation*}
		where $C_*$ may depend on $m,K,\nu,M$, but is independent of $J$ and $\varepsilon$.
	\end{lemma}

	\begin{proof}
		The construction begins by introducing a softened version of interval 
		indicators.  
		Let $\sigma(t)=\max\{t,0\}$ denote the ReLU activation and define
		\begin{equation*}
			\rho(t)=\sigma(t)-\sigma(t-1),
		\end{equation*}
		which is a continuous piecewise-linear function equal to $0$ for $t\le 0$
		and equal to $1$ for $t\ge 1$.  
		Hence $\rho$ can be implemented exactly by a fixed ReLU subnetwork of constant
		size and depth.
		
		\begin{figure}[h]
			\centering
			\begin{tikzpicture}[scale=1.1]
				\draw[->] (-0.3,0) -- (5.0,0) node[right] {$t$};
				\draw[->] (0,-0.2) -- (0,1.5);
				\coordinate (A) at (1,0);
				\coordinate (Ap) at (1.6,0);
				\coordinate (Bm) at (3.4,0);
				\coordinate (B) at (4,0);
				\draw[dashed] (A) -- (1,1.2);
				\draw[dashed] (B) -- (4,1.2);
				\node[below] at (A) {$a$};
				\node[below] at (B) {$b$};
				\draw[thick] (1,1.1) -- (4,1.1);
				\node at (2.5,1.3) {$\mathbf{1}_{[a,b]}(t)$};
				\draw[thick,red]
				(0.5,0) -- (1,0) -- (1.6,1) -- (3.4,1) -- (4,0) -- (4.5,0);
				\node[red] at (2.5,0.5) {$\chi^{(\theta)}_{[a,b]}(t)$};
				\draw[<->,blue] (1,0.2) -- (1.6,0.2);
				\node[blue] at (1.3,0.45) {$\theta h$};
				\draw[<->,blue] (3.4,0.2) -- (4,0.2);
				\node[blue] at (3.7,0.45) {$\theta h$};
			\end{tikzpicture}
			\caption{Soft indicator $\chi^{(\theta)}_{[a,b]}(t)$: it equals $1$ on 
				$[a+\theta h,b-\theta h]$, vanishes outside $[a,b]$, and has linear transitions 
				of width $\theta h$ near $a$ and $b$.}
		\end{figure}
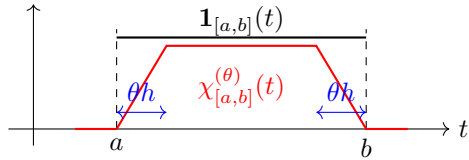
		
		Given an interval $[a,b]$ and a parameter $\theta\in(0,1/2)$,
		introduce the softened indicator
		\begin{equation*}
			\chi^{(\theta)}_{[a,b]}(t)
			=
			\rho\!\left(\frac{t-a}{\theta h}\right)
			-
			\rho\!\left(\frac{t-(b-\theta h)}{\theta h}\right),
		\end{equation*}
		which takes values in $[0,1]$, equals $1$ on the inner interval
		$[a+\theta h,b-\theta h]$, equals $0$ outside $[a,b]$, and undergoes linear 
		transition across boundary layers of width $\theta h$.
		
		Turning to cubes, each hypercube
		\begin{equation*}
			Q_j = \prod_{k=1}^m [a_{jk},b_{jk}]
		\end{equation*}
		admits the softened indicator
		\begin{equation*}
			\chi^{(\theta)}_{Q_j}(x)
			=
			\prod_{k=1}^m \chi^{(\theta)}_{[a_{jk},b_{jk}]}(x_k),
			\qquad x=(x_1,\dots,x_m).
		\end{equation*}
		Every factor lies in $[0,1]$, equals $1$ on points at distance at least 
		$\theta h$ from $\partial Q_j$, and equals $0$ outside $Q_j$.
		
		\begin{figure}[h]
			\centering
			\begin{tikzpicture}[scale=1.0]
				\draw[thick] (0,0) rectangle (4,4);
				\node at (0.2,-0.35) {$Q_j$};
				\draw[fill=white] (0.6,0.6) rectangle (3.4,3.4);
				\begin{scope}
					\clip (0,0) rectangle (4,4);
					\fill[blue!15] (0,0) rectangle (4,4);
					\fill[white] (0.6,0.6) rectangle (3.4,3.4);
				\end{scope}
				\draw[thick] (0,0) rectangle (4,4);
				\draw[thick,blue!60] (0.6,0.6) rectangle (3.4,3.4);
				\draw[<->,blue] (0.6,0.2) -- (0,0.2);
				\node[blue] at (0.3,0.45) {$\theta h$};
			\end{tikzpicture}
			\caption{For the cube $Q_j$, the soft indicator equals $1$ on the inner white 
				region and $0$ outside $Q_j$. The shaded band of thickness $\theta h$ inside 
				$Q_j$ is the boundary layer where the transition occurs.}
		\end{figure}
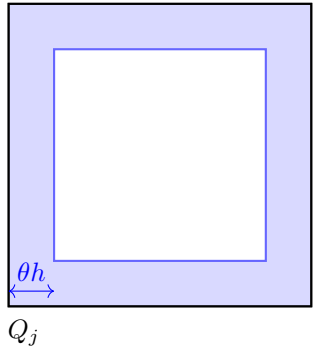
		
		Applying the approximate multiplication construction in
		\cite{Yar17} separately to the $m$-factor product
		defining each $\chi^{(\theta)}_{Q_j}$, one obtains, for every $j$,
		a ReLU subnetwork approximating $\chi^{(\theta)}_{Q_j}$ to accuracy
		$\delta\in(0,1)$ whose size and depth are both bounded by
		\begin{equation*}
		C_m\bigl(1+\log(\delta^{-1})\bigr),
	    \end{equation*}
		where $C_m$ depends only on $m$.
		
		If $M=0$ or $\nu(K)=0$, the assertion follows from the zero network.
		Hence assume that $M>0$ and $\nu(K)>0$. Using these soft indicators,
		define the softened piecewise-constant function
		\begin{equation*}
			\widetilde H_J(x)
			=
			\sum_{j=1}^J c_j\,\chi^{(\theta)}_{Q_j}(x),
		\end{equation*}
		and set
		\begin{equation*}
			\mathcal B_J:=\bigcup_{j=1}^J\partial Q_j.
		\end{equation*}
		For $\delta>0$, define
		\begin{equation*}
			U_\delta:=\{x\in K:\operatorname{dist}(x,\mathcal B_J)<\delta\}.
		\end{equation*}
		By assumption, $\nu(K\cap\mathcal B_J)=0$, and since $\mathcal B_J$ is
		closed, $U_\delta\downarrow K\cap\mathcal B_J$ as $\delta\downarrow0$.
		Hence there exists $\delta_\varepsilon\in(0,h/2)$ such that
		\begin{equation*}
			\nu(U_{\delta_\varepsilon})
			\le \left(\frac{\varepsilon}{4M}\right)^2.
		\end{equation*}
		Choose $\theta=\delta_\varepsilon/h\in(0,1/2)$. Then the discrepancy
		between $\widetilde H_J$ and $H_J$ is supported in
		$U_{\delta_\varepsilon}$, and because
		$|\widetilde H_J-H_J|\le 2M$, we obtain
		\begin{equation*}
			\|\widetilde H_J-H_J\|_{L^2(K,\nu)}
			\le
			2M\sqrt{\nu(U_{\delta_\varepsilon})}
			\le \varepsilon/2.
		\end{equation*}
		
		It remains to replace the ideal products $\chi^{(\theta)}_{Q_j}$ by
		ReLU-realizable approximations. Choose
		$
		\delta_{\mathrm{mul}}
		:=
		\min\left\{\frac12,\frac{\varepsilon}{2MJ\nu(K)^{1/2}}\right\}
		$.
		For each $j$, let $\widehat\chi_{Q_j}$ be a ReLU network approximation of $\chi^{(\theta)}_{Q_j}$ satisfying
		\begin{equation*}
			\|\widehat\chi_{Q_j}-\chi^{(\theta)}_{Q_j}\|_{L^\infty(K)}
			\le \delta_{\mathrm{mul}}.
		\end{equation*}
		Define
		\begin{equation*}
			\mathcal A_J(x)
			:=
			\sum_{j=1}^{J}c_j\widehat\chi_{Q_j}(x).
		\end{equation*}
		Then
		\begin{equation*}
		\begin{aligned}
			\|\mathcal A_J-\widetilde H_J\|_{L^2(K,\nu)}
			&\le
			\sum_{j=1}^{J}|c_j|\,
			\|\widehat\chi_{Q_j}-\chi^{(\theta)}_{Q_j}\|_{L^2(K,\nu)}  \\
			&\le
			M J\delta_{\mathrm{mul}}\nu(K)^{1/2}
			\le \varepsilon/2.
		\end{aligned}
		\end{equation*}
		Combining this with the boundary-layer estimate gives
		\begin{equation*}
			\|\mathcal A_J-H_J\|_{L^2(K,\nu)}
			\le
			\|\mathcal A_J-\widetilde H_J\|_{L^2(K,\nu)}
			+
			\|\widetilde H_J-H_J\|_{L^2(K,\nu)}
			\le \varepsilon.
		\end{equation*}
		Let
		\begin{equation*}
			\widehat\chi(x)
			:=
			\bigl(\widehat\chi_{Q_1}(x),\ldots,\widehat\chi_{Q_J}(x)\bigr)^{\top}
			\in\mathbb R^J
		\end{equation*}
		and $C_J:=[c_1,\ldots,c_J]\in\mathbb R^{p\times J}$. Then $\mathcal A_J(x)=C_J\widehat\chi(x)$. The $J$ scalar indicator subnetworks are placed in parallel, while the final affine map $C_J:\mathbb R^J\to\mathbb R^p$, taken with zero bias, contributes at most $pJ$ nonzero weights. Consequently,
		\begin{equation*}
			\mathrm{size}(\mathcal A_J)
			\le
			C_*J\bigl(1+\log(\delta_{\mathrm{mul}}^{-1})\bigr)+pJ,
			\qquad
			\mathrm{depth}(\mathcal A_J)
			\le
			C_*\bigl(1+\log(\delta_{\mathrm{mul}}^{-1})\bigr).
		\end{equation*}
		Since
		\begin{equation*}
			\delta_{\mathrm{mul}}^{-1}
			\le
			\max\left\{2,2M\nu(K)^{1/2}\right\}\frac{J}{\varepsilon},
		\end{equation*}
		there exists a constant, depending only on $M$ and $\nu(K)$, such that
		\begin{equation*}
			1+\log(\delta_{\mathrm{mul}}^{-1})
			\le
			C_*\left(1+\log\frac{J}{\varepsilon}\right).
		\end{equation*}
		Substituting this estimate into the preceding size and depth
		bounds, and enlarging $C_*$ if necessary, yields the estimates
		stated in the lemma. Together with
		$\|\mathcal A_J-H_J\|_{L^2(K,\nu)}\le\varepsilon$, this completes
		the proof.
	\end{proof}
	
	We next consider the corresponding approximation estimate under a Lipschitz stability assumption.
    \begin{theorem}[Lipschitz-stable comparison estimate]
		Fix $m,p\in\mathbb N$. The logarithmic estimate in Theorem~\ref{thm:abstract-approximation} may be compared with a setting in which the physical inverse map satisfies a Lipschitz stability estimate on the relevant admissible data class. More precisely, let $D_G^{\mathrm{Lip}}\subset D_G$ be the admissible data set generated by such a finite-dimensional admissible class, and assume that $G:D_G^{\mathrm{Lip}}\to Y$ satisfies
		\begin{equation*}
			\|G(f_1)-G(f_2)\|_Y
			\le
			C_{\mathrm{Lip}}\|f_1-f_2\|_X,
			\qquad f_1,f_2\in D_G^{\mathrm{Lip}}.
		\end{equation*}
		Let $\mathcal D:\mathbb R^m\to X$ and $\mathcal P:Y\to\mathbb R^p$ be Lipschitz continuous, and set $K:=\operatorname{supp}(\mathcal E_\#\mu)\subset\mathbb R^m$. Assume that $K$ is compact and that
		$\mathcal D(K)\subset D_G^{\mathrm{Lip}}$. Define $H:=\mathcal P\circ G\circ\mathcal D:K\to\mathbb R^p$.
		Then, for every $\eta\in(0,1)$, there exists a ReLU neural network $\mathcal A_\eta:\mathbb R^m\to\mathbb R^p$ such that
		\begin{equation*}
			\mathscr E_{\mathcal A}(\mathcal A_\eta)
			=
			\|\mathcal A_\eta-H\|_{L^2(K,\mathcal E_\#\mu)}
			\le \eta,
		\end{equation*}
		and
		\begin{equation*}
			\mathrm{size}(\mathcal A_\eta)
			\le c'_0(p+1)\eta^{-m}\bigl(1+\log(\eta^{-1})\bigr),
			\qquad
			\mathrm{depth}(\mathcal A_\eta)
			\le c'_1+c'_2m\log(\eta^{-1}).
		\end{equation*}
		The constants may depend on the prescribed ranks and the fixed data in the assumptions, but are independent of $\eta$ and $N$. Moreover, for all sufficiently large $N$,
		\begin{equation*}
			\mathscr E_{\mathcal A}(N)
			\le
			C\left(\frac{(p+1)\log N}{N}\right)^{1/m}.
		\end{equation*}
	\end{theorem}
	\begin{proof}
		Let $L_{\mathcal D}$ and $L_{\mathcal P}$ denote the Lipschitz constants of $\mathcal D$ and $\mathcal P$, respectively. Since $\mathcal D(K)\subset D_G^{\mathrm{Lip}}$, the map $H$ is Lipschitz on $K$ with
		\begin{equation*}
			\|H(x_1)-H(x_2)\|_{\mathbb R^p}
			\le
			L_H\|x_1-x_2\|_{\mathbb R^m},
			\qquad
			L_H:=L_{\mathcal P}C_{\mathrm{Lip}}L_{\mathcal D}.
		\end{equation*}
		Choose a cubical partition of a rectangle containing $K$ into $J=n^m$ cells of mesh width $h$ such that each cell boundary has $(\mathcal E_\#\mu)$-measure zero. For each cell $Q_j$ intersecting $K$, choose $x_j\in Q_j\cap K$ and define
		\begin{equation*}
			H_J(x):=H(x_j),
			\qquad x\in Q_j\cap K.
		\end{equation*}
		Since $\|x-x_j\|_{\mathbb R^m}\le\sqrt m\,h$ on each cell,
		\begin{equation*}
			\|H-H_J\|_{L^\infty(K)}
			\le L_H\sqrt m\,h.
		\end{equation*}
		Choosing $h\asymp\eta$ sufficiently small gives
		\begin{equation*}
			\|H-H_J\|_{L^\infty(K)}\le\eta/2,
			\qquad
			J\le C_J\eta^{-m}.
		\end{equation*}
		
		Applying Lemma~\ref{lem:PC-J} with tolerance $\eta/2$ and using
		$J\le C_J\eta^{-m}$, after enlarging the constants, yields a ReLU
		network $\mathcal A_\eta:\mathbb R^m\to\mathbb R^p$ such that
		\begin{equation*}
			\|\mathcal A_\eta-H_J\|_{L^2(K,\mathcal E_\#\mu)}
			\le\eta/2,
		\end{equation*}
		and
		\begin{equation*}
			\mathrm{size}(\mathcal A_\eta)
			\le
			c'_0(p+1)\eta^{-m}\bigl(1+\log(\eta^{-1})\bigr),
			\qquad
			\mathrm{depth}(\mathcal A_\eta)
			\le
			c'_1+c'_2m\log(\eta^{-1}).
		\end{equation*}
		Consequently,
		\begin{equation*}
			\|\mathcal A_\eta-H\|_{L^2(K,\mathcal E_\#\mu)}
			\le
			\|\mathcal A_\eta-H_J\|_{L^2(K,\mathcal E_\#\mu)}
			+
			\|H_J-H\|_{L^\infty(K)}
			\le\eta.
		\end{equation*}
		
		For sufficiently large $N$, set
		\begin{equation*}
			\eta_N
			:=
			\left(
			\frac{4c'_0(p+1)\log N}{N}
			\right)^{1/m}.
		\end{equation*}
		Then $\eta_N\in(0,1)$ and
		$1+\log(\eta_N^{-1})\le2\log N$. Hence
		\begin{equation*}
			\mathrm{size}(\mathcal A_{\eta_N})
			\le
			c'_0(p+1)
			\frac{N}{4c'_0(p+1)\log N}
			\,2\log N
			\le N.
		\end{equation*}
		Therefore,
		\begin{equation*}
			\mathscr E_{\mathcal A}(N)
			\le\eta_N
			\le
			C\left(
			\frac{(p+1)\log N}{N}
			\right)^{1/m},
		\end{equation*}
		which completes the proof.
	\end{proof}
	\begin{remark}
		For prescribed encoder and reconstruction dimensions $m$ and $p$, the derived upper bounds exhibit different dependence on the target accuracy: the logarithmic-stability construction yields an exponential network-size bound, whereas the Lipschitz-stability construction yields an algebraic bound. This is consistent with the stronger ill-posedness of the logarithmically stable inverse problem.
	\end{remark}
    
    \subsection{Finite-dimensional reconstruction}\label{subsec:reconstruction}
    
    The reconstruction module is a finite-dimensional, possibly affine, approximation mechanism for the output space $Y$.
    Let $\Pi_p^Y : Y \to Y$ be a finite-dimensional approximation map admitting a factorization
    \begin{equation*}
    \Pi_p^Y = \mathcal R_p \circ \mathcal P_p
    \end{equation*}
    for some Borel measurable maps
    \begin{equation*}
    \mathcal P_p : Y \to \mathbb{R}^p,
    \qquad
    \mathcal R_p : \mathbb{R}^p \to Y.
    \end{equation*}
    
    \begin{theorem}[Abstract reconstruction principle]
    	For any reconstruction mechanism of the form
    	\begin{equation*}
    	\Pi_p^Y = \mathcal R_p \circ \mathcal P_p,
    	\end{equation*}
    	one has
    	\begin{equation*}
    	\mathscr E_{\mathcal R}
    	=
    	\left(
    	\int_Y \|y-\Pi_p^Y y\|_Y^2\, d((\mathscr G)_\#\mu)(y)
    	\right)^{1/2}.
    	\end{equation*}
    	In particular, if
    	\begin{equation*}
    	\int_Y \|y-\Pi_p^Y y\|_Y^2\, d((\mathscr G)_\#\mu)(y)
    	\le \varepsilon_R(p)^2,
    	\end{equation*}
    	then
    	\begin{equation*}
    	\mathscr E_{\mathcal R}\le \varepsilon_R(p).
    	\end{equation*}
    \end{theorem}
    
    \begin{proof}
    	By the factorization $\Pi_p^Y = \mathcal R_p \circ \mathcal P_p$,
    	\begin{equation*}
    	\mathscr E_{\mathcal R}
    	=
    	\left(
    	\int_Y \|\mathcal R_p\circ \mathcal P_p(y)-y\|_Y^2\, d((\mathscr G)_\#\mu)(y)
    	\right)^{1/2}
    	=
    	\left(
    	\int_Y \|\Pi_p^Y y-y\|_Y^2\, d((\mathscr G)_\#\mu)(y)
    	\right)^{1/2}.
    	\end{equation*}
    \end{proof}
    
    We next describe several reconstruction mechanisms covered by this abstract principle.
    \setcounter{method}{0}
    \begin{method}{Smoothness-based spectral reconstruction}
    	A natural reconstruction mechanism is obtained from a spectral basis of the output space.
    	Suppose that $Y^s \hookrightarrow Y$ continuously, and let $\{\phi_j\}_{j\ge 1}$ be a spectral basis of $Y$.
    	Define
    	\[
    	\Pi_p^Y y := \sum_{j=1}^p \langle y,\phi_j\rangle_Y \phi_j,
    	\qquad
    	\mathcal P_p(y):=\bigl(\langle y,\phi_1\rangle_Y,\dots,\langle y,\phi_p\rangle_Y\bigr),
    	\qquad
    	\mathcal R_p(c):=\sum_{j=1}^p c_j\phi_j.
    	\]
    	Then
    	\[
    	\mathcal R_p\circ \mathcal P_p=\Pi_p^Y.
    	\]
    	If
    	\[
    	\|y-\Pi_p^Y y\|_Y \le C a_p \|y\|_{Y^s},
    	\qquad y\in Y^s,
    	\]
    	and
    	\[
    	\int_X \|\mathscr G(x)\|_{Y^s}^2\, d\mu(x)\le M_Y,
    	\]
    	then
    	\[
    	\mathscr E_{\mathcal R}^2
    	=
    	\int_Y \|y-\Pi_p^Y y\|_Y^2\, d((\mathscr G)_\#\mu)(y)
    	\le
    	C^2 a_p^2
    	\int_X \|\mathscr G(x)\|_{Y^s}^2\, d\mu(x),
    	\]
    	hence
    	\[
    	\mathscr E_{\mathcal R}\le C M_Y^{1/2} a_p.
    	\]
    	This is the reconstruction mechanism used when one has regularity information on the image of $\mathscr G$.
    \end{method}
    
    \begin{method}{Covariance-spectral reconstruction}
    	A second canonical mechanism is induced by the covariance operator of the pushforward measure. Assume that $\mathscr G_\#\mu$ has finite second moment, and define its mean by
    	\begin{equation*}
    		\bar y_\mu:=\int_Y y\,d(\mathscr G_\#\mu)(y).
    	\end{equation*}
    	Let $\Gamma_{\mathscr G_\#\mu}$ be the covariance operator of $\mathscr G_\#\mu$, with eigenpairs $(\sigma_j,\xi_j)$. Define
    	\begin{equation*}
    		\Pi_p^Y y
    		=
    		\sum_{j=1}^p \langle y-\bar y_\mu,\xi_j\rangle_Y\,\xi_j+\bar y_\mu,
    	\end{equation*}
    	\begin{equation*}
    		\mathcal P_p(y):=
    		\bigl(
    		\langle y-\bar y_\mu,\xi_1\rangle_Y,\dots,
    		\langle y-\bar y_\mu,\xi_p\rangle_Y
    		\bigr),
    		\qquad
    		\mathcal R_p(c):=\sum_{j=1}^p c_j\xi_j+\bar y_\mu.
    	\end{equation*}
    	Then $\mathcal R_p\circ \mathcal P_p=\Pi_p^Y$, and
    	\begin{equation*}
    		\int_Y \|y-\Pi_p^Y y\|_Y^2\,d(\mathscr G_\#\mu)(y)
    		=
    		\sum_{j>p}\sigma_j.
    	\end{equation*}
    	Hence
    	\begin{equation*}
    		\mathscr E_{\mathcal R}^2=\sum_{j>p}\sigma_j.
    	\end{equation*}
    	This is the covariance-based spectral reconstruction mechanism, also known as Karhunen--Lo\`eve or principal component analysis (PCA) reconstruction \cite{LMK22}. This mechanism is optimal at the affine level, but may be difficult to quantify explicitly for nonlinear $\mathscr G$, since the eigensystem of $\mathscr G_\#\mu$ is generally hard to characterize.
    \end{method}
    
    \begin{method}{Affine reconstruction generated by trunk functions}
    	Suppose that $Y$ is a function space on an output domain $\Omega$.
    	Let $\tau_1,\ldots,\tau_p\in Y$ be prescribed reconstruction
    	functions and let $b\in Y$. Define
    	\[
    	\mathcal R_p(c):=b+\sum_{j=1}^p c_j\tau_j,
    	\qquad c=(c_1,\ldots,c_p)\in\mathbb R^p.
    	\]
    	For a coefficient map $\mathcal P_p:Y\to\mathbb R^p$, set
    	\[
    	\Pi_{p,\tau}^Y:=\mathcal R_p\circ \mathcal P_p.
    	\]
    	The image of $\Pi_{p,\tau}^Y$ is contained in the affine space
    	$b+\operatorname{span}\{\tau_1,\ldots,\tau_p\}$.
    	Consequently,
    	\[
    	\mathscr E_{\mathcal R}
    	=
    	\left(
    	\int_Y
    	\|y-\Pi_{p,\tau}^Y y\|_Y^2
    	\,d(\mathscr G_{\#}\mu)(y)
    	\right)^{1/2}.
    	\]
    	Thus, for fixed trunk functions and coefficient map $\mathcal P_p$, the reconstruction error is precisely the mean-square error associated with the induced affine approximation map $\Pi_{p,\tau}^Y$.
    \end{method}

    \begin{method}{Trunk-net realization of ideal reconstruction bases}
    	A finite-rank reconstruction may first be defined by an ideal affine reconstructor
    	\begin{equation*}
    		\mathcal R_p^*(c):=\sum_{j=1}^p c_j\phi_j^*+b^*,
    		\qquad c\in\mathbb R^p,
    	\end{equation*}
    	together with a coefficient map $\mathcal P_p:Y\to\mathbb R^p$. The corresponding ideal affine approximation is
    	\begin{equation*}
    		\Pi_p^{Y,*}:=\mathcal R_p^*\circ\mathcal P_p.
    	\end{equation*}
    	In an implemented DeepONet, the branch network produces the coefficients $c_j$, while the trunk network realizes the reconstruction functions in the output variable. More precisely, when $Y$ is a function space on an output domain $\Omega$, the vector-valued trunk map $\boldsymbol{\tau}:=(\tau_1,\ldots,\tau_p):\Omega\to\mathbb R^p$ is realized by a feedforward neural network. Thus the ideal functions $\phi_j^*$ may be replaced by trunk functions $\tau_j$, giving
    	\begin{equation*}
    		\mathcal R_p(c):=\sum_{j=1}^p c_j\tau_j+b.
    	\end{equation*}
    	Let $\nu:=\mathscr G_\#\mu$. Then
    	\begin{equation*}
    		\mathscr E_{\mathcal R}
    		=
    		\left(
    		\int_Y
    		\|y-\mathcal R_p\mathcal P_p y\|_Y^2
    		\,d\nu(y)
    		\right)^{1/2}
    		\le
    		\mathscr E_{\mathcal R}^{\mathrm{rank}}
    		+
    		\mathscr E_{\mathcal R}^{\mathrm{trunk}},
    	\end{equation*}
    	where
    	\begin{equation*}
    		\mathscr E_{\mathcal R}^{\mathrm{rank}}
    		:=
    		\left(
    		\int_Y
    		\|y-\mathcal R_p^*\mathcal P_p y\|_Y^2
    		\,d\nu(y)
    		\right)^{1/2},
    	\end{equation*}
    	and
    	\begin{equation*}
    		\mathscr E_{\mathcal R}^{\mathrm{trunk}}
    		:=
    		\left(
    		\int_Y
    		\|(\mathcal R_p^*-\mathcal R_p)\mathcal P_p y\|_Y^2
    		\,d\nu(y)
    		\right)^{1/2}.
    	\end{equation*}
    \end{method}
	
	\section{DeepONet error estimate}\label{section3}
	In this section, we combine the approximation result of Section~\ref{section2} with the encoding and reconstruction errors to derive an error estimate for the full DeepONet.
	\begin{theorem}[DeepONet error estimate]\label{thm:error-decomposition}
		Fix $m,p\in\mathbb N$, and let $c_0,c_1,c_2,c_3,N_0$ be as in Theorem~\ref{thm:abstract-approximation}. Assume the logarithmic stability estimate in Assumption~\ref{ass:stability}. Let $\mu$ be a Borel probability measure on $X$ such that $\mu(D_G)=1$, and let
		\begin{equation*}
			K:=\operatorname{supp}(\mathcal E_\#\mu)\subset \mathbb R^m
		\end{equation*}
		be compact. Assume that $\mathcal D(K)\subset D_G$, that $\mathcal D:\mathbb R^m\to X$ and $\mathcal P:Y\to\mathbb R^p$ are Lipschitz continuous, and that
		\begin{equation*}
			\varepsilon_m
			:=
			\operatorname*{ess\,sup}_{x\in X}
			\|\mathcal D\circ\mathcal E(x)-x\|_X
			\le \min\{\tau,e^{-1}\}.
		\end{equation*}
		The essential supremum is taken with respect to $\mu$, and hence $\mathscr E_{\mathcal E}\le\varepsilon_m$ since $\mu$ is a probability measure. Assume further that $\mathcal R$ is Lipschitz continuous. In addition, suppose that, for some prescribed reconstruction rate $\varepsilon_R(p)$, the chosen reconstruction mechanism satisfies $\mathscr E_{\mathcal R} \le C_2 \varepsilon_R(p)$ for some constant $C_2>0$ independent of $p$.
		
		Then, for every $N>N_0$, there exists a ReLU approximator
		$\mathcal A_N:\mathbb R^m\to\mathbb R^p$, with
		$\mathrm{size}(\mathcal A_N)\le N$, such that
		\begin{equation}\label{total-error}
			\mathscr E
			\le
			\Lip(\mathcal R)
			\left[
			\frac{1}{c_1m}
			\log\!\left(\frac{N}{c_0(p+1)}\right)
			\right]^{-\beta}
			+
			\Lip(\mathcal R\circ\mathcal P)\,C_{\mathrm{stab}}|\log\varepsilon_m|^{-\beta}
			+
			C_2\varepsilon_R(p).
		\end{equation}
		Moreover,
			for every $\eta\in(0,1)$, the construction provides an approximator satisfying
			\begin{equation*}
				\mathrm{size}(\mathcal A_\eta)
				\le c_0(p+1)\exp \bigl(c_1m\eta^{-1/\beta}\bigr),
				\qquad
				\mathrm{depth}(\mathcal A_\eta)
				\le c_2+c_3m\eta^{-1/\beta}.
			\end{equation*}
	\end{theorem}
	
	\begin{proof}
		We decompose the total error into three parts:
		\begin{equation*}
			\begin{aligned}
				\mathscr{N} - \mathscr{G}
				&= \mathcal{R}\circ\mathcal{A}\circ\mathcal{E} - \mathscr{G} \\
				&= [\mathcal{R}\circ\mathcal{A}\circ\mathcal{E} - \mathcal{R}\circ\mathcal{P}\circ \mathscr{G}]
				+[\mathcal{R}\circ\mathcal{P}\circ \mathscr{G}-\mathscr{G}] \\
				&= [\mathcal{R}\circ\mathcal{A}\circ\mathcal{E} - \mathcal{R}\circ\mathcal{P}\circ \mathscr{G}\circ\mathcal{D}\circ\mathcal{E}] \\
				&\quad + [\mathcal{R}\circ\mathcal{P}\circ \mathscr{G}\circ\mathcal{D}\circ\mathcal{E} - \mathcal{R}\circ\mathcal{P}\circ \mathscr{G}] \\
				&\quad + [\mathcal{R}\circ\mathcal{P}\circ \mathscr{G} - \mathscr{G}] \\
				&=: T_1+T_2+T_3.
			\end{aligned}
		\end{equation*}
		
		We can now estimate the norms of these three terms as follows:
		\begin{equation*}
			\begin{aligned}
				\|T_1\|_{L^2(\mu)}
				&=
				\left(
				\int_X
				\|\mathcal{R}\circ\mathcal{A}\circ\mathcal{E}
				-
				\mathcal{R}\circ\mathcal{P}\circ \mathscr{G}\circ\mathcal{D}\circ\mathcal{E}\|_{Y}^{2}
				\,d\mu
				\right)^{1/2}\\
				&\leq
				\Lip(\mathcal{R})
				\left(
				\int_X
				\|\mathcal{A}\circ\mathcal{E}
				-
				\mathcal{P}\circ \mathscr{G}\circ\mathcal{D}\circ\mathcal{E}\|_{\ell^2(\RR^p)}^{2}
				\,d\mu
				\right)^{1/2}\\
				&=
				\Lip(\mathcal{R})
				\left(
				\int_{\RR^m}
				\|
				\mathcal{A}
				-
				\mathcal{P}\circ \mathscr{G}\circ\mathcal{D}
				\|_{\ell^2(\RR^p)}^{2}
				\,d(\mathcal{E}_{\#}\mu)
				\right)^{1/2}\\
				&=
				\Lip(\mathcal{R})\cdot \mathscr{E}_{\mathcal{A}}.
			\end{aligned}
		\end{equation*}
		
		For the second term, we get
		\begin{equation*}
			\begin{aligned}
				\|T_2\|_{L^2(\mu)}
				&=
				\left(
				\int_X
				\|
				\mathcal{R}\circ\mathcal{P}\circ \mathscr{G}\circ\mathcal{D}\circ\mathcal{E}
				-
				\mathcal{R}\circ\mathcal{P}\circ \mathscr{G}
				\|_{Y}^{2}
				\,d\mu
				\right)^{1/2}\\
				&\leq
				\Lip(\mathcal{R}\circ\mathcal{P})
				\left(
				\int_X
				\|
				\mathscr{G}\circ\mathcal{D}\circ\mathcal{E}
				-
				\mathscr{G}
				\|_{Y}^{2}
				\,d\mu
				\right)^{1/2}.
			\end{aligned}
		\end{equation*}
		
		For $\mu$-a.e. $x$, both $x$ and $\mathcal D\circ\mathcal E(x)$ belong to $D_G$. Applying the logarithmic stability estimate,
		\begin{equation*}
			\|\mathscr{G}\circ\mathcal{D}\circ\mathcal{E}(x)-\mathscr{G}(x)\|_Y
			\leq
			C_{\mathrm{stab}}\,\omega(\|\mathcal{D}\circ\mathcal{E}(x)-x\|_X),
		\end{equation*}
		where $\omega(t)=|\log t|^{-\beta}$.
		Let
		\begin{equation*}
			\delta(x)=\|\mathcal{D}\circ\mathcal{E}(x)-x\|_X,
		\end{equation*} 
		\begin{equation*}
			\varepsilon_m:=\operatorname*{ess\,sup}_{x\in X}\|\mathcal{D}\circ\mathcal{E}(x)-x\|_X,
		\end{equation*}
		where the essential supremum is taken with respect to $\mu$.
		By definition of $\varepsilon_m$, $\delta(x)\le \varepsilon_m$ for $\mu$-a.e. $x$. Since $\varepsilon_m\le \min\{\tau,e^{-1}\}$, the logarithmic stability estimate applies on the relevant range and $\omega$ is monotone there. Hence
		\begin{equation*}
			\omega(\delta(x))
			\leq
			\omega(\varepsilon_m)
			=
			|\log \varepsilon_m|^{-\beta}.
		\end{equation*}
		Therefore,
		\begin{equation*}
			\left(
			\int_X (\omega(\delta(x)))^2\,d\mu(x)
			\right)^{1/2}
			\leq
			\omega(\varepsilon_m)
			=
			|\log \varepsilon_m|^{-\beta}.
		\end{equation*}
		We obtain
		\begin{equation*}
			\|T_2\|_{L^2(\mu)}
			\leq
			\Lip(\mathcal{R}\circ\mathcal{P})\cdot C_{\mathrm{stab}}\cdot |\log \varepsilon_m|^{-\beta}.
		\end{equation*}
		
		For the third term, we obtain
		\begin{equation*}
			\begin{aligned}
				\|T_3\|_{L^2(\mu)}
				&=
				\left(
				\int_X
				\|
				\mathcal{R}\circ\mathcal{P}\circ \mathscr{G}
				-
				\mathscr{G}
				\|_{Y}^{2}
				\,d\mu
				\right)^{1/2}\\
				&=
				\left(
				\int_Y
				\|(\mathcal R\circ\mathcal P)(y)-y\|_Y^2
				\,d(\mathscr G_\#\mu)(y)
				\right)^{1/2}\\
				&=
				\mathscr E_{\mathcal R}.
			\end{aligned}
		\end{equation*}
		By the assumed reconstruction error estimate, we have
		\begin{equation*}
			\|T_3\|_{L^2(\mu)}
			\leq
			C_2\,\varepsilon_R(p).
		\end{equation*}
		
		Combining these estimates yields
		\begin{equation*}
			\mathscr{E}
			\leq
			\Lip(\mathcal{R})\cdot \mathscr{E}_{\mathcal{A}}
			+
			\Lip(\mathcal{R}\circ\mathcal{P})\cdot C_{\mathrm{stab}}\cdot |\log \varepsilon_m|^{-\beta}
			+
			C_2\,\varepsilon_R(p).
		\end{equation*}
		
		Finally, Theorem~\ref{thm:abstract-approximation} provides an approximator $\mathcal A_N$ of size at most $N$ satisfying
		\begin{equation*}
			\mathscr E_{\mathcal A}
			\le
			\left[
			\frac{1}{c_1m}
			\log\!\left(\frac{N}{c_0(p+1)}\right)
			\right]^{-\beta}.
		\end{equation*}
		Substitution into the preceding decomposition gives \eqref{total-error}.
	\end{proof}
	
	\begin{remark}
		In the Lipschitz-stable regime discussed above, assume in addition that $\mu(D_G^{\mathrm{Lip}})=1$. Then, for all sufficiently large $N$, the same error decomposition yields
		\begin{equation*}
			\mathscr E
			\le
			\Lip(\mathcal R)\,C_1
			\left(\frac{(p+1)\log N}{N}\right)^{1/m}
			+
			\Lip(\mathcal R\circ\mathcal P)\,C_{\mathrm{Lip}}\varepsilon_m
			+
			C_2\varepsilon_R(p).
		\end{equation*}
	\end{remark}
	
	\section{Application to the inverse scattering problem}\label{section4}
	
	In this section, we apply the abstract DeepONet framework to inverse acoustic medium scattering. Let $d\in\{2,3\}$ and $B:=\{x\in\mathbb R^d:|x|\le R\}$. We consider a complex-valued medium contrast $q\in L^\infty(\mathbb R^d)$ satisfying $\operatorname{supp}(q)\Subset B^\circ$. Unless otherwise stated, all scalar-valued function spaces in this section and the appendices are understood over $\mathbb C$. The scattering of a time-harmonic acoustic plane wave by the corresponding inhomogeneous medium is described by the total field $u_q=u^{\mathrm{inc}}+u_q^{\mathrm{sca}}$, which satisfies
	
	\begin{equation}\label{eq:1.1}
		\Delta u_{q}+k^{2}(1+q) u_{q}=0 \text{ in } \mathbb{R}^{d},
	\end{equation}
	and the Sommerfeld radiation condition 
	\begin{equation}\label{eq:1.2}
		\lim _{|x| \to \infty}|x|^{\frac{d-1}{2}}\left(\frac{\partial u_{q}^{\mathrm{sca}}}{\partial|x|}-i k u_{q}^{\mathrm{sca} }\right)=0.
	\end{equation}
	
	Assume that $u^{\mathrm{inc}}$ is the plane incident field, i.e., $u^{\mathrm{inc}}=e^{i k x \cdot \omega}$ with $\omega \in \mathbb{S}^{d-1}$. Then the scattered field $u_{q}^{\mathrm{sca}}$ satisfies 
	\begin{equation}\label{eq:uqsca}
		u_{q}^{\mathrm{sca}}(x, \omega)=\frac{e^{i k|x|}}{|x|^{\frac{d-1}{2}}} u_{q}^{\infty}(\hat{x}, \omega)+O\left(|x|^{-\frac{d+1}{2}}\right) \text{ as } |x| \to \infty,
	\end{equation}
	where $\hat{x}=x /|x|$; see, for example, \cite[p.~232]{Ser17}. The inverse scattering problem is to determine the medium contrast $q$ from knowledge of the far-field pattern or scattering amplitude $u_q^\infty(\hat{x},\omega)$ for all $\hat{x}$, $\omega \in \mathbb{S}^{d-1}$ at one fixed energy $k^{2}$. Inverse problems for identifying inhomogeneous media and obstacles have been studied extensively; see, for example, \cite{CK19, CC14, CCH22, KG07, NP15}. For a computational perspective, we refer to \cite{Che18}.
	
	It is known that the far-field pattern $u_{q}^{\infty}(\hat{x}, \omega)$ uniquely determines the near-field data of \eqref{eq:1.1} on $\partial B$, which in turn determines the Dirichlet-to-Neumann map of \eqref{eq:1.1} provided $0$ is not a Dirichlet eigenvalue of $\Delta+k^{2}(1+q)$ on $B$, see \cite{Nac88}. Combining this fact with the uniqueness results proved in \cite{SU87, Buk08} implies that $u_{q}^{\infty}(\hat{x}, \omega)$ for all $\hat{x}, \omega \in \mathbb{S}^{d-1}$ determines $q$ uniquely, at least when $q$ is essentially bounded. The inverse scattering problem is notoriously ill-posed. Correspondingly, a log-type estimate has been derived in \cite{HH01}. Under additional finite-dimensional or structural restrictions, Lipschitz stability estimates are available; see \cite{Bou13,AS22}.

	\subsection{DeepONet setup}
	
	For the qualitative operator-learning formulation, we restrict the contrast to the closed admissible class $\mathcal Q\subset C(B)$ satisfying the assumptions of Theorem~\ref{borel-measurability}. For $q\in\mathcal Q$, let $u_{q}^{\infty}(\hat{x},\omega)$ be the corresponding far-field pattern. Let $D:=\{u_q^\infty:q\in\mathcal Q\}\subset C(\mathbb S^{d-1}\times\mathbb S^{d-1})$ be the physical far-field data range. We use a Borel measurable extension
	\[
		\mathscr G:L^2(\mathbb S^{d-1}\times\mathbb S^{d-1})\to L^2(B)
	\]
	of the physical inverse map, satisfying
	$q=\mathscr{G}(u_{q}^{\infty}(\hat{x},\omega))$ for every $q\in\mathcal Q$; see Theorem~\ref{borel-measurability}.
	It is known that $u_{q}^{\infty}(\hat{x},\omega)$ is analytic on $\mathbb{S}^{d-1}\times\mathbb{S}^{d-1}$, and hence smooth. Our aim here is to construct a surrogate for $\mathscr{G}$ using a DeepONet architecture.
	
	To begin, let $\mathcal{E}:C(\mathbb{S}^{d-1}\times\mathbb{S}^{d-1}) \to\mathbb{C}^{n\times n}$ be an \textit{encoder}. Since the far-field pattern is complex-valued, we identify
	\(\mathbb C^{n\times n}\) with \(\mathbb R^{2n^2}\) by separating real and imaginary parts. Thus the output of the encoder is regarded as a real vector in \(\mathbb R^m\), with \(m=2n^2\). For complex-valued contrasts $q$, if $r_p$ complex reconstruction coefficients are retained, their real and imaginary parts are concatenated, so that $\mathbb C^{r_p}\simeq\mathbb R^p$ with $p=2r_p$; the reconstructor then recombines these real coordinates into complex coefficients.
	Next, we define an \textit{approximator} $\mathcal{A}:\mathbb{R}^{m}\to\mathbb{R}^{p}$ for some $p\in\mathbb{N}$, which is a feedforward deep neural network.
	
	As in the DeepONet architecture, we define the \textit{branch net} as the composition of $\mathcal{A}$ and $\mathcal{E}$, i.e., $\boldsymbol{\beta}=\mathcal{A}\circ\mathcal{E}:C(\mathbb{S}^{d-1}\times\mathbb{S}^{d-1})\to\mathbb{R}^{p}$. A \textit{reconstructor} $\mathcal{R}:\mathbb{R}^{p}\to C(B)$ is then defined.
	Putting all components together, we now have a DeepONet:
	
	\begin{equation*}
		\mathscr{N}:C(\mathbb{S}^{d-1}\times\mathbb{S}^{d-1})\to C(B)\subset L^{2}(B),\quad\mathscr{N}(f)=(\mathcal{R}\circ\mathcal{A}\circ\mathcal{E})(f).
	\end{equation*}
	
	To describe the approximation error of $\mathscr{N}$ to $\mathscr{G}$ rigorously, let $\mu$ be a Borel probability measure on $X:=L^{2}(\mathbb{S}^{d-1}\times\mathbb{S}^{d-1})$ satisfying $\mu(D)=1$. Since $D\subset C(\mathbb{S}^{d-1}\times\mathbb{S}^{d-1})$, this implies $\mu(C(\mathbb{S}^{d-1}\times\mathbb{S}^{d-1}))=1$. Assume that $\mu$ has a finite second moment, i.e., $\int_{X}\|f\|_{X}^{2}d\mu(f)<\infty$. We now define
	
	\begin{equation}\label{error}
		\mathscr{E}:=\left(\int_{X}\int_{B}|\mathscr{G}(f)(y)-\mathscr{N}(f)(y)|^{2}dy d\mu(f)\right)^{1/2}.
	\end{equation}
	
	Since the encoder $\mathcal{E}$ takes the pointwise values of $u_{q}^{\infty}(\hat{x},\omega)$ on $\mathbb{S}^{d-1}\times\mathbb{S}^{d-1}$, it is not well-defined for functions in $L^{2}(\mathbb{S}^{d-1}\times\mathbb{S}^{d-1})$. Following \cite[Lemma B.1]{LMK22}, we use a Borel measurable extension of the pointwise encoder to $X=L^2(\mathbb S^{d-1}\times\mathbb S^{d-1})$; the resulting DeepONet $\mathscr N=\mathcal R\circ\mathcal A\circ\mathcal E$ is therefore Borel measurable on $X$, and \eqref{error} is well defined. We can directly apply \cite[Theorem 3.1]{LMK22} to prove a universal approximation theorem for the mapping $\mathscr{G}$, which maps from the far-field pattern to the perturbation $q$. The following qualitative approximation result shows that the inverse-scattering map can be approximated by a DeepONet in the $L^2(\mu;L^2(B))$ sense.

	\begin{theorem}[DeepONet approximation for inverse scattering]\label{thm:deeponet-inverse-scattering-approx}
		For any $\varepsilon > 0$, there exists a DeepONet $\mathscr{N} = \mathcal{R} \circ \mathcal{A} \circ \mathcal{E}: C(\mathbb{S}^{d-1} \times \mathbb{S}^{d-1}) \to L^2(B)$ whose Borel measurable extension to $X$, still denoted by $\mathscr N$, satisfies
		\begin{equation*}
			\| \mathscr{G} - \mathscr{N} \|_{L^2(\mu)} = \left( \int_X \int_B | \mathscr{G}(f)(y) - \mathscr{N}(f)(y) |^2  dy  d\mu(f) \right)^{1/2} < \varepsilon,
		\end{equation*}
		where $X = L^2(\mathbb{S}^{d-1} \times \mathbb{S}^{d-1})$, $Y = L^2(B)$, and $\mu$ is a Borel probability measure on $X$ satisfying $\mu(D)=1$. Since $D\subset C(\mathbb{S}^{d-1} \times \mathbb{S}^{d-1})$, this implies $\mu(C(\mathbb{S}^{d-1} \times \mathbb{S}^{d-1}))=1$. Assume that $\mu$ has a finite second moment. Assume that the restricted measure $\mu|_{C(\mathbb{S}^{d-1} \times \mathbb{S}^{d-1})}$ is a Borel probability measure on
		$\big(C(\mathbb{S}^{d-1} \times \mathbb{S}^{d-1}),\|\cdot\|_{\infty}\big)$, and $\int_X \|\mathscr G(f)\|_{L^2(B)}^2\,d\mu(f)<\infty$.
	\end{theorem}
	\begin{proof}
		See Appendix~\ref{app:qualitative-approximation}.
	\end{proof}
	
	Theorem~\ref{thm:deeponet-inverse-scattering-approx} is qualitative. We next derive quantitative estimates that explicitly separate the encoder, neural approximation, and reconstruction errors. For the quantitative analysis, we now specialize the error decomposition of Section~\ref{section3} using problem-adapted encoding, admissible lifting, and reconstruction mechanisms. 
	
	Before deriving the quantitative estimates, we clarify the choice of encoder. The qualitative approximation result above uses a point-evaluation encoder, whereas the quantitative analysis below uses a spectral coefficient encoder to exploit the analytic regularity of the far-field data. The linear synthesis operator $\mathcal S_L$ is used only to represent the orthogonal projection $P_L$ and quantify its truncation error, while the admissibility-preserving decoder $\mathcal D_L^{\mathrm{ad}}$ is used in the physical stability argument. In particular, $P_Lf$ need not itself be an admissible far-field pattern.
	\subsection{Quantitative error estimates}
	We now derive the quantitative encoding, reconstruction, and neural approximation estimates needed for the inverse-scattering error bound. We begin with the spectral approximation of the far-field data, which underlies the encoding error estimate.
		\begin{lemma}
		\label{lem:sph-harm-approx}
		For notational clarity, we state the angular approximation result in the
		three-dimensional spherical-harmonic notation.  The two-dimensional case is
		obtained in the same way by replacing the spherical-harmonic basis on
		$\mathbb S^2$ with the Fourier basis on
		$\mathbb S^1$.
		
		Let $P_L$ denote the orthogonal projection onto the tensor-product subspace spanned by spherical harmonics of degree $\le L$ in each variable.
		Then there exist constants $C,c>0$, independent of $f$ and $L$, such that, for every $f\in\mathcal X$,
		\begin{equation*}
			\|f-P_L f\|_{L^2(\mathbb{S}^{d-1}\times\mathbb{S}^{d-1})}
			\le Ce^{-cL}\,\|f\|_{\mathcal X},
		\end{equation*}
		where $\mathcal X$ is a fixed analytic class continuously embedded in $L^2(\mathbb S^{d-1}\times\mathbb S^{d-1})$ such that the tensor spherical-harmonic coefficients of every $f\in\mathcal X$ satisfy
		\begin{equation*}
			|\hat f_{l_1,m_1;l_2,m_2}|
			\le C\,\|f\|_{\mathcal X}e^{-c(l_1+l_2)},
			\qquad l_1,l_2\ge0.
		\end{equation*}
		Such exponential spherical-harmonic decay is standard for analytic
		classes; see, for example, \cite{VV18}.
	\end{lemma}
	
	\begin{proof}
		Expand $f$ in the tensor spherical-harmonic basis:
		\begin{equation*}
		f(\hat x,\omega)
		=\sum_{l_1,l_2\ge0}\sum_{m_1=1}^{d_{l_1}}\sum_{m_2=1}^{d_{l_2}}
		\hat f_{l_1,m_1;l_2,m_2}\,
		Y_{l_1,m_1}(\hat x)\,Y_{l_2,m_2}(\omega),
		\end{equation*}
		where $d_l=N(d,l)=O(l^{d-2})$.
		
		By the exponential-decay property, there exist constants $C_1,\sigma>0$, depending only on the fixed analytic class, such that
		\begin{equation}\label{eq:exp-decay}
			|\hat f_{l_1,m_1;l_2,m_2}|
			\le
			C_1\|f\|_{\mathcal X}
			e^{-\sigma(l_1+l_2)}.
		\end{equation}
		
		Since $P_L$ removes all terms with $l_1>L$ or $l_2>L$, Parseval's identity gives
		\begin{equation*}
		\|f-P_L f\|_{L^2}^2
		=\sum_{\substack{l_1,l_2\ge0\\ l_1>L\text{ or }l_2>L}}
		\sum_{m_1,m_2}|\hat f_{l_1,m_1;l_2,m_2}|^2.
		\end{equation*}
			Using \eqref{eq:exp-decay} and $d_l=O(l^{d-2})$, fix any
			$\sigma'\in(0,\sigma)$. Then
			\begin{equation*}
				\sum_{l>L} d_l e^{-2\sigma l}
				\le
				C\sum_{l>L}(1+l)^{d-2}e^{-2\sigma l}
				\le
				C_{\sigma'}e^{-2\sigma' L},
			\end{equation*}
			where $C_{\sigma'}>0$ is independent of $L$. Moreover,
			\begin{equation*}
				\sum_{l\ge0}d_l e^{-2\sigma l}<\infty.
			\end{equation*}
			Consequently,
			\begin{equation*}
				\begin{aligned}
					\|f-P_L f\|_{L^2}^2
					&\le
					C_1^2\|f\|_{\mathcal X}^2
					\left[
					\left(\sum_{l_1>L}d_{l_1}e^{-2\sigma l_1}\right)
					\left(\sum_{l_2\ge0}d_{l_2}e^{-2\sigma l_2}\right)
					\right.\\
					&\qquad\left.
					+
					\left(\sum_{l_1\ge0}d_{l_1}e^{-2\sigma l_1}\right)
					\left(\sum_{l_2>L}d_{l_2}e^{-2\sigma l_2}\right)
					\right] \\
					&\le
					C_2e^{-2\sigma' L}\|f\|_{\mathcal X}^2 .
				\end{aligned}
			\end{equation*}
			Taking square roots gives
			\begin{equation*}
				\|f-P_L f\|_{L^2}
				\le
				C e^{-\sigma' L}\|f\|_{\mathcal X}.
			\end{equation*}
			Thus the stated estimate follows with $c=\sigma'$.
	\end{proof}
	
	The preceding spectral approximation estimate yields the following encoding error bound under a uniform analytic regularity condition.
	\begin{theorem}[Spectral encoding error]
		\label{prop:encoder-error}
		Let $\mathcal X$ be the analytic class in
		Lemma~\ref{lem:sph-harm-approx}, and assume that there exists
		$M_{\mathcal X}>0$ such that
		\begin{equation*}
			\|f\|_{\mathcal X}\le M_{\mathcal X}
			\qquad \text{for }\mu\text{-a.e. } f.
		\end{equation*}
		Then there exist constants $C_1,c_1>0$ such that the spectral synthesis error satisfies
		\begin{equation*}
		\varepsilon_{m_L}^{\mathrm{spec}}
		:=
		\operatorname*{ess\,sup}_{f\in X}
		\|\mathcal S_L\mathcal E_L(f)-f\|_{L^2(\mathbb S^{d-1}\times\mathbb S^{d-1})}
		\le C_1 e^{-c_1 m_L^{1/(2(d-1))}} ,
		\end{equation*}
		where the essential supremum is taken with respect to $\mu$. Here $\mathcal E_L$ is the truncated spherical-harmonic coefficient encoder and $\mathcal S_L$ is the corresponding linear spectral synthesis operator. In particular, to guarantee $\varepsilon_{m_L}^{\mathrm{spec}}\le \varepsilon$, it suffices to choose the real encoder dimension $m_L$ such that
		\begin{equation*}
		m_L \ge \left(\frac{1}{c_1}\log\frac{C_1}{\varepsilon}\right)^{2(d-1)}.
		\end{equation*}
	\end{theorem}

	\begin{proof}
		Let
		\begin{equation*}
		r_L
		:=
		\left(\sum_{l=0}^{L}d_l\right)^2,
		\qquad
		m_L:=2r_L.
		\end{equation*}
		Using the identification $\mathbb C^{r_L}\simeq\mathbb R^{m_L}$
		obtained by separating real and imaginary parts, define
		\begin{equation*}
		\mathcal E_L(f)
		:=
		\bigl(
		\langle f,Y_{l_1,m_1}(\hat x)Y_{l_2,m_2}(\omega)\rangle
		\bigr)_{\max(l_1,l_2)\le L}
		\in\mathbb C^{r_L}\simeq\mathbb R^{m_L}.
		\end{equation*}
		For
		$c=(c_{l_1,m_1;l_2,m_2})
		\in\mathbb C^{r_L}\simeq\mathbb R^{m_L}$, define
		\begin{equation*}
		\mathcal S_L(c)
		:=
		\sum_{\max(l_1,l_2)\le L}
		c_{l_1,m_1;l_2,m_2}
		Y_{l_1,m_1}(\hat x)Y_{l_2,m_2}(\omega).
		\end{equation*}
		By orthonormality of the tensor-product spherical-harmonic basis,
		\begin{equation*}
		\mathcal S_L\mathcal E_L(f)
		=
		\sum_{\max(l_1,l_2)\le L}
		\langle f,Y_{l_1,m_1}Y_{l_2,m_2}\rangle
		Y_{l_1,m_1}Y_{l_2,m_2}
		=P_Lf.
		\end{equation*}
		Thus $\mathcal S_L\mathcal E_L$ realizes the orthogonal projection $P_L$. We emphasize that $\mathcal S_L$ is only a linear synthesis map; it is not the admissible physical decoder used in the stability argument below.

		The number of retained complex coefficients satisfies
		\begin{equation*}
		r_L
		=
		\left(\sum_{l=0}^{L}d_l\right)^2
		\asymp L^{2(d-1)};
		\end{equation*}
		see \cite{DX13}. Hence the real encoder dimension satisfies
		\begin{equation*}
		m_L=2r_L\asymp L^{2(d-1)},
		\qquad
		L\asymp m_L^{1/(2(d-1))}.
		\end{equation*}
		By Lemma~\ref{lem:sph-harm-approx} and the assumed uniform bound,
		\begin{equation*}
		\|\mathcal S_L\mathcal E_L(f)-f\|_{L^2}
		=\|P_Lf-f\|_{L^2}
		\le Ce^{-cL}\|f\|_{\mathcal X}
		\le C_1e^{-c_1L}
		\end{equation*}
		for $\mu$-a.e. $f$. Taking the essential supremum and substituting $L\asymp m_L^{1/(2(d-1))}$ gives the stated estimate.
	\end{proof}
	
    Having quantified the spectral encoding error, we next turn to the reconstruction component. The following lemma provides the polynomial approximation estimate underlying the spectral reconstruction error.
	\begin{lemma}\label{lem:ball-polynomial-approx}
		For notational clarity, we use spherical-harmonic--Jacobi notation on the
		$d$-dimensional ball.  When $d=2$, this is understood as the corresponding
		Fourier--Jacobi basis on the disk, obtained by replacing the angular spherical
		harmonics on $\mathbb S^1$ with Fourier modes.
		
		Let $q\in H^s(B)$ for some $s>0$.
		Suppose that $\{\phi_1,\phi_2,\dots\}$ is a real-valued orthonormal
		spherical-harmonic--Jacobi polynomial basis of $L^2(B;\mathbb R)$.
		For $n\in\mathbb N$, set
		\begin{equation*}
		V_{r_p}
		:=
		\operatorname{span}_{\mathbb C}\{\phi_1,\dots,\phi_{r_p}\}
		=
		\mathbb P_n(B),
		\end{equation*}
		where
		$
		r_p
		=
		\dim_{\mathbb C}\mathbb P_n(B)
		\asymp n^d,
		p:=2r_p
		$.
		
		Let $P_{r_p}:L^2(B)\to V_{r_p}$ be the complex
		orthogonal projection defined by
		\begin{equation*}
		P_{r_p}q
		=
		\sum_{k=1}^{r_p}
		\langle q,\phi_k\rangle_{L^2(B)}\,\phi_k.
		\end{equation*}
		Then there exists a constant $C=C(B,d,s)>0$ such that
		\begin{equation*}
		\|q-P_{r_p}q\|_{L^2(B)}
		\le
		Cp^{-s/d}\|q\|_{H^s(B)}.
		\end{equation*}
	\end{lemma}
	
	\begin{proof}
		The orthogonal projection $P_{r_p}$ satisfies the best-approximation
		property
		\begin{equation*}
		\|q-P_{r_p}q\|_{L^2(B)}
		=
		\inf_{v\in V_{r_p}}\|q-v\|_{L^2(B)}.
		\end{equation*}
		Applying the Jackson-type inequality on the ball \cite{DL93}
		separately to $\operatorname{Re}q$ and $\operatorname{Im}q$ gives
		\begin{equation*}
		\|q-P_{r_p}q\|_{L^2(B)}
		\le
		C_1n^{-s}\|q\|_{H^s(B)},
		\end{equation*}
		where $C_1=C_1(B,d,s)$. Since $p=2r_p\asymp n^d$, there
		exist constants $C_2',C_2>0$, depending only on $d$, such that
		\begin{equation*}
		C_2'n^d\le p\le C_2n^d.
		\end{equation*}
		Consequently,
		\begin{equation*}
		n^{-s}\le C_2^{s/d}p^{-s/d},
		\end{equation*}
		and the stated estimate follows after absorbing $C_2^{s/d}$ into the constant.
	\end{proof}
	
	The preceding polynomial approximation estimate yields the following reconstruction error bound.
	\begin{theorem}[Spectral reconstruction error]\label{thm:inverse-scattering-reconstruction}
		Assume that the pushforward measure 
		$\mathscr{G}_{\#}\mu$ is supported on $H^s(B)$ and satisfies
			$\|q\|_{H^s(B)}\le C_q$ for all
			$q\in\operatorname{supp}(\mathscr{G}_{\#}\mu)$.
		
			Let $r_p$ be the number of retained complex reconstruction modes
			and set $p:=2r_p$. Suppose that the reconstructor
			$\mathcal{R}_{\mathrm{SH}}$ and the projector $\mathcal P$ are chosen
			so that
			\begin{equation*}
				\mathcal R_{\mathrm{SH}}\circ\mathcal P=P_{r_p},
			\end{equation*}
			where $P_{r_p}$ is the complex orthogonal projection onto
			\begin{equation*}
				V_{r_p}
				=
				\operatorname{span}_{\mathbb C}
				\{\phi_1,\dots,\phi_{r_p}\}
				=
				\mathbb P_n(B),
				\qquad
				r_p\asymp n^d.
			\end{equation*}
			Here $\{\phi_k\}_{k=1}^{r_p}$ are the first $r_p$ real-valued
			spherical-harmonic--Jacobi basis functions described above.
		
		Then there exists a constant $C=C(B,d,s)>0$ such that
			\begin{equation*}
				\mathscr{E}_{\mathcal{R}}
				\le C\,C_q\,p^{-s/d}.
			\end{equation*}
			In particular, to guarantee
			$\mathscr E_{\mathcal R}\le \varepsilon$, it suffices to choose the
			even real reconstruction dimension $p$ such that
			\begin{equation*}
				p \ge
				\left(\frac{C\,C_q}{\varepsilon}\right)^{d/s}.
			\end{equation*}
		
	\end{theorem}
	
	\begin{proof}
			For $q\in L^2(B)$, define the bounded real-linear
			projector $\mathcal P:L^2(B)\to\mathbb R^p$ by
			\begin{equation*}
				\mathcal P(q)
				:=
				\bigl(
				\operatorname{Re}\langle q,\phi_1\rangle,\dots,
				\operatorname{Re}\langle q,\phi_{r_p}\rangle,
				\operatorname{Im}\langle q,\phi_1\rangle,\dots,
				\operatorname{Im}\langle q,\phi_{r_p}\rangle
				\bigr).
			\end{equation*}
			For $(a,b)\in\mathbb R^{r_p}\times\mathbb R^{r_p}
			\simeq\mathbb R^p$, define
			\begin{equation*}
				\mathcal R_{\mathrm{SH}}(a,b)
				:=
				\sum_{k=1}^{r_p}(a_k+ib_k)\phi_k.
			\end{equation*}
			Thus,
			\begin{equation*}
				\mathcal R_{\mathrm{SH}}\circ\mathcal P(q)
				=
				\sum_{k=1}^{r_p}
				\langle q,\phi_k\rangle_{L^2(B)}\phi_k
				=
				P_{r_p}q.
			\end{equation*}
			By Lemma~\ref{lem:ball-polynomial-approx},
			\begin{equation*}
				\|q-P_{r_p}q\|_{L^2(B)}
				\le C\,p^{-s/d}\|q\|_{H^s(B)}
				\le C\, C_q\, p^{-s/d}.
			\end{equation*}
			Thus,
			\begin{equation*}
				\mathscr{E}_{\mathcal{R}}^2
				= \int_Y \|q-P_{r_p}q\|_{L^2(B)}^2
				\,d(\mathscr{G}_{\#}\mu)(q)
				\le C^2\, C_q^2\, p^{-2s/d}\int_Y d(\mathscr{G}_{\#}\mu)(q)
				= C^2 C_q^2 p^{-2s/d}.
			\end{equation*}
			Taking square roots gives
			\begin{equation*}
				\mathscr{E}_{\mathcal{R}} \le C\, C_q\, p^{-s/d}.
			\end{equation*}
			
	\end{proof}
	
	We next state the logarithmic stability estimate for the inverse scattering problem that will be used in the subsequent quantitative analysis. Let $B_1:=\{x\in\mathbb R^3:|x|<1\}$ denote the unit ball.
	\begin{lemma}[\cite{HH01}]\label{lem:ln}
		Let $s>3/2$, $C_q>0$, and
		$0<\epsilon<\frac{s}{s+3}$ be fixed. Then there exists a positive constant $C$ (depending only on $s$, $\epsilon$, $k$, and $C_q$) such that for all
		$q,\widetilde q\in H^s(\mathbb R^3)$ satisfying
		\begin{equation*}
			\|q\|_{H^s},\|\widetilde q\|_{H^s}\le C_q,
			\qquad
			\operatorname{supp}(q),\operatorname{supp}(\widetilde q)\subset B_1,
			\qquad
			\operatorname{Im}q,\operatorname{Im}\widetilde q\ge0
			\quad\text{a.e.},
		\end{equation*}
		the following stability estimate holds:
		\begin{equation*}
		\|q - \widetilde{q}\|_{L^2(B_1)} \leq C \left[ -\ln^{-} \left( \|u_q^\infty - u_{\widetilde{q}}^\infty\|_{L^2(\mathbb{S}^2 \times \mathbb{S}^2)} \right) \right]^{-\frac{s}{s+3} + \epsilon},
		\end{equation*}
		where $u_q^\infty$ and $u_{\widetilde{q}}^\infty$ denote the far-field patterns corresponding to $q$ and $\widetilde{q}$, respectively. Here $\ln^{-}(0):=-\infty$, while $\ln^{-}(t):=\ln(t)$ for $0<t\leq\exp(-1)$ and $\ln^{-}(t):=-1$ for $t>\exp(-1)$, with $[-\ln^{-}(0)]^{-\alpha}:=0$ for $\alpha>0$.
	\end{lemma}
	
	For the quantitative estimates based on Lemma~\ref{lem:ln}, we now restrict the inverse problem to the admissible class
	\begin{equation*}
		\mathcal Q_{\mathrm{ad}}
		:=
		\left\{
		q\in H^s(\mathbb R^3):
		\operatorname{supp}(q)\subset B_1,\,
		\|q\|_{H^s}\le C_q,\,
		\operatorname{Im}q\ge0\ \text{a.e.}
		\right\},
		\qquad s>\frac32.
	\end{equation*}
	Let
	\begin{equation*}
	\Psi:\mathcal Q_{\mathrm{ad}}\to L^2(\mathbb S^2\times\mathbb S^2),
	\qquad
	\Psi(q)=u_q^\infty,
    \end{equation*}
	and denote the corresponding admissible far-field data set by
	\begin{equation*}
	D_G:=\Psi(\mathcal Q_{\mathrm{ad}}).
	\end{equation*}
	On this set, the physical inverse map is
	\begin{equation*}
	G:D_G\to L^2(B_1),
	\qquad
	G(u_q^\infty)=q .
	\end{equation*}
	Thus Lemma~\ref{lem:ln} gives a logarithmic stability estimate for $G$ on $D_G$, with exponent $\hat s:=\frac{s}{s+3}-\epsilon \in (0,1)$.
	
	In the quantitative arguments below, the logarithmic stability estimate is applied only to pairs of admissible far-field data in $D_G$. Since a finite-dimensional spectral projection of the far-field data need not itself belong to the physical data set, we introduce an admissibility-preserving lifting mechanism that allows the stability estimate to be applied to finite-dimensional encoded representations.

	\paragraph{Prior-induced admissible decoder.}
	Equip $\mathcal Q_{\mathrm{ad}}$ with a fixed separable metric topology for which the inclusion
	$\mathcal Q_{\mathrm{ad}}\hookrightarrow L^2(B_1)$ and $\Psi$ are continuous. Let $\nu$ be a Borel probability measure on this space and assume that
	\[
	\mathcal Q_{\mathrm{pr}}:=\operatorname{supp}\nu
	\]
	is compact. The set $\mathcal Q_{\mathrm{pr}}$ represents the admissible prior class considered in the quantitative analysis. We set
	\[
	\mu:=\Psi_{\#}\nu .
	\]

		Since $\mu$ is supported on
		$\Psi(\mathcal Q_{\mathrm{pr}})\subset D_G$, we choose
		$\mathscr G$ to agree with the physical inverse map $G$ on
		$\Psi(\mathcal Q_{\mathrm{pr}})$. Hence all quantitative
		$L^2(\mu)$ errors are independent of the values of $\mathscr G$ outside this set.
	For the
	coefficient encoder $\mathcal E_L$, define
	\[
	K_L := (\mathcal E_L\circ \Psi)(\mathcal Q_{\mathrm{pr}}).
	\]
	By compactness and continuity, $K_L$ is compact and coincides with
	$\operatorname{supp}((\mathcal E_L)_{\#}\mu)$. Let
	\[
	F_L := \mathcal E_L\circ \Psi : \mathcal Q_{\mathrm{pr}}\to K_L .
	\]
	For every $z\in K_L$, the fiber $F_L^{-1}(z)$ is nonempty and compact, and
	the graph
	\[
	\Gamma_L
	:=
	\{(z,q)\in K_L\times \mathcal Q_{\mathrm{pr}}: F_L(q)=z\}
	\]
	is closed. Since $F_L$ is continuous and $\mathcal Q_{\mathrm{pr}}$ is a compact metric space, the multifunction $z\mapsto F_L^{-1}(z)$ is weakly measurable. Hence, by the Kuratowski--Ryll-Nardzewski measurable selection
	theorem, there exists a Borel measurable selector
	\[
	s_L:K_L\to \mathcal Q_{\mathrm{pr}}
	\]
	such that
	\[
	F_L(s_L(z))=z,\qquad z\in K_L .
	\]
	We write $q_z:=s_L(z)$ and define the prior-induced admissible decoder
	\[
	\mathcal D_L^{\mathrm{ad}}(z):=\Psi(q_z)=u^\infty_{q_z}.
	\]
	Thus
	\[
	\mathcal D_L^{\mathrm{ad}}(K_L)
	\subset \Psi(\mathcal Q_{\mathrm{pr}})
	\subset D_G .
	\]
	This is the decoder used in the physical stability argument. The linear
	spectral synthesis operator $\mathcal S_L$ from Theorem~\ref{prop:encoder-error} is used only to express
	$P_L f=\mathcal S_L\mathcal E_L f$ and to quantify the spectral truncation error; it is
	not inserted into the physical inverse map $G$. No continuity or Lipschitz
	continuity of $\mathcal D_L^{\mathrm{ad}}$ is assumed or used.
	
	\begin{lemma}[Spectral accuracy of the prior-induced admissible decoder]
	\label{lem:admissible-lifting}
	Assume that $\mathcal Q_{\mathrm{pr}}\subset \mathcal Q_{\mathrm{ad}}$ and that the prior-generated far-field manifold satisfies the uniform analytic bound
	\begin{equation*}
	\sup_{q\in\mathcal Q_{\mathrm{pr}}}\|\Psi(q)\|_{\mathcal X}\le M_{\mathcal X}^{\mathrm{pr}}.
	\end{equation*}
	Let $\mathcal D_L^{\mathrm{ad}}:K_L\to D_G$ be the prior-induced admissible decoder defined above. Then $\mathcal D_L^{\mathrm{ad}}(K_L)\subset D_G$, and
	\begin{equation*}
		\varepsilon_m^{\mathrm{ad}}
		:=
		\operatorname*{ess\,sup}_{f\in\Psi(\mathcal Q_{\mathrm{pr}})}
		\left\|
		\mathcal D_L^{\mathrm{ad}}\mathcal E_L(f)-f
		\right\|_{L^2(\mathbb S^2\times\mathbb S^2)}
		\le Ce^{-cL},
	\end{equation*}
	where the essential supremum is taken with respect to
	$\mu|_{\Psi(\mathcal Q_{\mathrm{pr}})}$.
	In particular, in three dimensions, $m_L\asymp L^4$ and therefore
	\begin{equation*}
	\varepsilon_m^{\mathrm{ad}}\le Ce^{-c m_L^{1/4}}.
	\end{equation*}
	\end{lemma}

	\begin{proof}
	The inclusion $\mathcal D_L^{\mathrm{ad}}(K_L)\subset D_G$ follows immediately from the construction: $\mathcal D_L^{\mathrm{ad}}(z)=\Psi(q_z)$ with $q_z\in \mathcal Q_{\mathrm{pr}}\subset\mathcal Q_{\mathrm{ad}}$.
	Let $f\in\Psi(\mathcal Q_{\mathrm{pr}})$. Then
	$f=\Psi(q)$ for some $q\in\mathcal Q_{\mathrm{pr}}$, and hence
	\begin{equation*}
		\mathcal E_L(f)
		=
		(\mathcal E_L\circ\Psi)(q)
		\in K_L.
	\end{equation*}
	Therefore,
	\begin{equation*}
		g:=\mathcal D_L^{\mathrm{ad}}\mathcal E_L(f)
	\end{equation*}
	is well defined.
	By construction, $g\in D_G$ and $\mathcal E_L g=\mathcal E_L f$. Hence $f$ and $g$ have the same spherical-harmonic coefficients up to degree $L$, so $P_Lg=P_Lf$. Consequently,
	\begin{equation*}
	\|g-f\|_{L^2}
	\le \|g-P_Lg\|_{L^2}+\|P_Lf-f\|_{L^2}.
	\end{equation*}
	The spectral-tail estimate in Lemma~\ref{lem:sph-harm-approx}, together with the uniform analytic bound on $\Psi(\mathcal Q_{\mathrm{pr}})$, gives
	\begin{equation*}
	\|g-P_Lg\|_{L^2}+\|P_Lf-f\|_{L^2}\le Ce^{-cL}.
	\end{equation*}
	Taking the essential supremum with respect to $\mu=\Psi_{\#}\nu$ proves the estimate. Since $m_L\asymp L^4$ for $d=3$, the final bound follows.
	\end{proof}
	
The following quantitative estimates provide a problem-specific admissible-lifting realization of the abstract framework in Section~\ref{section3}. The admissible decoder is used only in the stability proof and is not an additional computational module in the implemented network.
	
	\begin{theorem}[Quantitative ReLU approximation of the admissibly lifted inverse map]
	\label{thm:A-approximation}
	Under the assumptions of Lemma~\ref{lem:admissible-lifting},
	we specialize to the three-dimensional setting of the logarithmic stability estimate used below. Fix $L,p\in\mathbb N$. Let
	\begin{equation*}
	G:D_G\to L^2(B_1)
	\end{equation*}
	be the physical inverse map on the admissible far-field data set $D_G$, which satisfies the logarithmic stability estimate in
	Lemma~\ref{lem:ln}. Let
	\begin{equation*}
	\mathcal D_L^{\mathrm{ad}}:K_L\to D_G
	\end{equation*}
	be the prior-induced admissible decoder defined above, let
	$\mathcal P:L^2(B_1)\to\mathbb R^p$ be a bounded
	real-linear map, and define
	\begin{equation*}
	H_L:=\mathcal P\circ G\circ\mathcal D_L^{\mathrm{ad}}:K_L\to\mathbb R^p.
	\end{equation*}
	Then, for every $\eta\in(0,1)$, there exists a ReLU neural network $\mathcal A_\eta:\mathbb R^{m_L}\to\mathbb R^p$ such that
	\begin{equation*}
	\|\mathcal A_\eta-H_L\|_{L^2(K_L,(\mathcal E_L)_{\#}\mu)}
	\le C\eta+CL^{-\hat s},
	\end{equation*}
	and
	\begin{equation*}
	\mathrm{size}(\mathcal A_\eta)
	\le c_0(p+1)\exp \bigl(c_1m_L\eta^{-1/\hat s}\bigr),
	\qquad
	\mathrm{depth}(\mathcal A_\eta)
	\le c_2+c_3m_L\eta^{-1/\hat s}.
	\end{equation*}
	The constants $c_0,c_1,c_2,c_3>0$ may depend on $L$, $p$, and the fixed data in the assumptions, but are independent of $\eta$ and $N$. The constant $C>0$ in the approximation-error bounds is independent of $L$, $\eta$, and $N$, but may depend on the operator norm $\|\mathcal P\|$ and on the fixed data in the assumptions. Set
	\begin{equation*}
	N_{0,L}:=c_0(p+1)\exp(c_1m_L).
	\end{equation*}
	Then, for every $N> N_{0,L}$,
	\begin{equation}\label{eq:scattering-approximation-error}
	\mathscr E_{\mathcal A}(N)
	\le
	C\left[
	\frac{1}{c_1m_L}
	\log\!\left(\frac{N}{c_0(p+1)}\right)
	\right]^{-\hat s}
	+Cm_L^{-\hat s/4},
	\end{equation}
	where $\mathscr E_{\mathcal A}(N)$ is defined with target map $H_L$ and latent set $K_L$.
	\end{theorem}

	\begin{proof}
	The admissibility of the composition is a consequence of Lemma~\ref{lem:admissible-lifting}: for every $z\in K_L$, one has $\mathcal D_L^{\mathrm{ad}}(z)\in D_G$, so $H_L$ is well defined on $K_L$. We do not assume that $\mathcal D_L^{\mathrm{ad}}$ is Lipschitz. Instead, we derive the required modulus of continuity for $H_L$ directly from the spectral encoding.
	Let $z_1,z_2\in K_L$ and set
	\begin{equation*}
	g_i:=\mathcal D_L^{\mathrm{ad}}(z_i)\in D_G,\qquad i=1,2.
	\end{equation*}
	Since $\mathcal E_Lg_i=z_i$ and $P_L=\mathcal S_L\mathcal E_L$, orthonormality gives
	\begin{equation*}
		\|P_L(g_1-g_2)\|_{L^2}
		=
		\|z_1-z_2\|_{\ell^2}.
	\end{equation*}
	The high-frequency tails are controlled by the analytic estimate:
	\begin{equation*}
	\|(I-P_L)g_i\|_{L^2}\le Ce^{-cL},\qquad i=1,2.
	\end{equation*}
	Therefore
	\begin{equation*}
	\|g_1-g_2\|_{L^2(\mathbb S^2\times\mathbb S^2)}
	\le C\|z_1-z_2\|_{\ell^2}+Ce^{-cL}.
	\end{equation*}
	Because $g_1,g_2\in D_G$, the logarithmic stability estimate in Lemma~\ref{lem:ln} applies. Applying $\mathcal P$, we obtain
	\begin{equation*}
	\|H_L(z_1)-H_L(z_2)\|_{\mathbb R^p}
	\le C\left[-\ln^{-}\left(C\|z_1-z_2\|_{\ell^2}+Ce^{-cL}\right)\right]^{-\hat s}.
	\end{equation*}
	Choose a shifted cubical partition of a rectangle containing $K_L$ into cubes of side length $h$ such that every cell boundary has $(\mathcal E_L)_{\#}\mu$-measure zero, and define the corresponding piecewise-constant approximation $H_J$ on $K_L$. Since two points in the same cell are at most $\sqrt{m_L}h$ apart,
	\begin{equation*}
		\|H_L-H_J\|_{L^\infty(K_L)}
		\le
		C\left[
		-\ln^{-}\left(C\sqrt{m_L}h+Ce^{-cL}\right)
		\right]^{-\hat s}.
	\end{equation*}
	Choose $h\asymp m_L^{-1/2}\exp(-c\eta^{-1/\hat s})$, with the implicit constants independent of $L$ and $\eta$. Then
	\begin{equation*}
		\left[
		-\ln^{-}\left(C\sqrt{m_L}h+Ce^{-cL}\right)
		\right]^{-\hat s}
		\le C\eta+CL^{-\hat s},
	\end{equation*}
	and the number of cubes satisfies
	\begin{equation*}
		J
		\le
		C_L\exp\bigl(Cm_L\eta^{-1/\hat s}\bigr),
	\end{equation*}
	where $C_L>0$ is independent of $\eta$ and can be absorbed into
	the $L$-dependent constants in the network-size estimate.
	Applying Lemma~\ref{lem:PC-J} with tolerance $\eta$ yields a ReLU network satisfying
	\begin{equation*}
	\|\mathcal A_\eta-H_L\|_{L^2(K_L,(\mathcal E_L)_{\#}\mu)}
	\le C\eta+CL^{-\hat s},
	\end{equation*}
	and
	\begin{equation*}
	\mathrm{size}(\mathcal A_\eta)
	\le C_*J\left(1+\log\frac{J}{\eta}\right)+pJ,
	\qquad
	\mathrm{depth}(\mathcal A_\eta)
	\le C_*\left(1+\log\frac{J}{\eta}\right).
	\end{equation*}
	Combining these estimates with the bound for $J$ gives the stated size and depth bounds. Since $m_L\asymp L^4$ in three dimensions, $L^{-\hat s}\asymp m_L^{-\hat s/4}$.
	For $N> N_{0,L}$, define
	\begin{equation*}
	\eta_{N,L}
	:=
	\left[
	\frac{1}{c_1m_L}
	\log\!\left(\frac{N}{c_0(p+1)}\right)
	\right]^{-\hat s}.
	\end{equation*}
	Then $\eta_{N,L}\in(0,1)$ and $\mathrm{size}(\mathcal A_{\eta_{N,L}})\le N$, and substitution of $\eta_{N,L}$ into the preceding approximation estimate gives \eqref{eq:scattering-approximation-error}.
	\end{proof}
 
    Combining the preceding encoding, neural approximation, and reconstruction estimates yields the following quantitative DeepONet error bound for inverse acoustic scattering.
	\begin{theorem}[DeepONet error bound for inverse acoustic scattering]
	\label{thm:main}
	Under the assumptions of Lemma~\ref{lem:ln} and
	\ref{lem:admissible-lifting}, fix $L,n\in\mathbb N$, set
	\begin{equation*}
		r_p:=\dim_{\mathbb C}\mathbb P_n(B_1),
		\qquad
		m=m_L,
		\qquad
		p=2r_p.
	\end{equation*}
	Let $\nu$ be the prior measure introduced above, so that
	$\mathcal Q_{\mathrm{pr}}=\operatorname{supp}\nu
	\subset\mathcal Q_{\mathrm{ad}}$, and set
	$\mu=\Psi_{\#}\nu$. Take $\mathcal E_L$ to be the spectral coefficient encoder and
	\begin{equation*}
		\mathcal D_L^{\mathrm{ad}}:K_L\to D_G
	\end{equation*}
	to be the prior-induced admissible decoder defined above.
		Let $\mathcal P$ and $\mathcal R=\mathcal R_{\mathrm{SH}}$ be the
		spectral coefficient projector and reconstructor specified in
		Theorem~\ref{thm:inverse-scattering-reconstruction}, specialized
		to $d=3$ and $B=B_1$, so that $\mathcal R\circ\mathcal P=P_{r_p}$.
		By orthonormality,
		$\|\mathcal P\|\le 1$, $\Lip(\mathcal R)=1$, and
		$\Lip(\mathcal R\circ\mathcal P)\le 1$, uniformly in $p$.
	Assume that $m$ is sufficiently large so that $\varepsilon_m^{\mathrm{ad}}$ lies in the small-data regime of the logarithmic stability estimate. Let $c_0,c_1,N_{0,L}$ be as in Theorem~\ref{thm:A-approximation}. Then there exist constants $C_1,C_2,C_3>0$, independent of $m,p$, and $N$, such that, for every $N>N_{0,L}$, there exists a ReLU approximator $\mathcal A_N:\mathbb R^m\to\mathbb R^p$ of size at most $N$ for which the corresponding DeepONet satisfies
	\begin{equation}\label{eq:inverse-scattering-error}
	\begin{aligned}
	\mathscr E
	\le{}&
	C_1\Lip(\mathcal R)
	\left[
	\frac{1}{c_1m}
	\log\!\left(\frac{N}{c_0(p+1)}\right)
	\right]^{-\hat s}\\
	&+C_2\bigl(\Lip(\mathcal R)+\Lip(\mathcal R\circ\mathcal P)\bigr)m^{-\hat s/4}
	+C_3p^{-s/3},
	\end{aligned}
	\end{equation}
	where $\hat s=\frac{s}{s+3}-\epsilon$ and $0<\epsilon<\frac{s}{s+3}$.
	\end{theorem}

	\begin{proof}
	By Lemma~\ref{lem:admissible-lifting}, the decoder $\mathcal D_L^{\mathrm{ad}}$ is admissibility preserving:
	\begin{equation*}
	\mathcal D_L^{\mathrm{ad}}(K_L)\subset D_G.
	\end{equation*}
	Consequently, for $\mu$-a.e.
	$f\in\Psi(\mathcal Q_{\mathrm{pr}})$, both $f$ and
	$\mathcal D_L^{\mathrm{ad}}\mathcal E_L(f)$ belong to $D_G$. Using the same three-term algebraic decomposition as in the proof of Theorem~\ref{thm:error-decomposition},
	\begin{equation*}
	\mathscr E
	\le \Lip(\mathcal R)\mathscr E_{\mathcal A}
	+\Lip(\mathcal R\circ\mathcal P)C|\log\varepsilon_m^{\mathrm{ad}}|^{-\hat s}
	+\mathscr E_{\mathcal R}.
	\end{equation*}
	The admissible lifting estimate in Lemma~\ref{lem:admissible-lifting} yields
	\begin{equation*}
	\varepsilon_m^{\mathrm{ad}}\le Ce^{-cm^{1/4}},
	\end{equation*}
	and hence, for $m$ sufficiently large,
	\begin{equation*}
	|\log\varepsilon_m^{\mathrm{ad}}|^{-\hat s}\le Cm^{-\hat s/4}.
	\end{equation*}
	The approximation term is bounded by Theorem~\ref{thm:A-approximation}:
	\begin{equation*}
	\mathscr E_{\mathcal A}
	\le
	C\left[
	\frac{1}{c_1m}
	\log\!\left(\frac{N}{c_0(p+1)}\right)
	\right]^{-\hat s}
	+Cm^{-\hat s/4},
	\end{equation*}
	and the reconstruction term is bounded by Theorem~\ref{thm:inverse-scattering-reconstruction}:
	\begin{equation*}
	\mathscr E_{\mathcal R}\le Cp^{-s/3}.
	\end{equation*}
	Combining the three estimates gives \eqref{eq:inverse-scattering-error}.
	\end{proof}

	\subsection{Numerical experiments}

We test the proposed DeepONet construction on inverse acoustic medium scattering in two and three spatial dimensions. In each dimension, we consider an infinite-dimensional prior on a function space and a finite-dimensional prior supported on a prescribed trial space. Results are reported separately for the two- and three-dimensional settings. Because the two prior classes define different reconstruction tasks and use different output representations, their error levels are not interpreted as a controlled model comparison. The experiments assess reconstruction accuracy on independent test data and the empirical sensitivity of the trained models to perturbations of the far-field measurements.

\subsubsection{Forward solver and mitigation of the inverse crime}

Let $B_d$ denote the unit disk when $d=2$ and the unit ball when $d=3$.  For an incident direction $\omega\in\SSS^{d-1}$, the incident wave is
\begin{equation*}
	u^{\mathrm{inc}}(x,\omega)=e^{ikx\cdot\omega}.
\end{equation*}
Given a compactly supported contrast $q$, the total field satisfies
\begin{equation*}
	\Delta u_q+k^2(1+q)u_q=0\qquad \text{in }\RR^d,
\end{equation*}
together with the Sommerfeld radiation condition.  We solve the forward problem through the Lippmann--Schwinger equation
\begin{equation}\label{eq:num_ls}
	u_q(x,\omega)=u^{\mathrm{inc}}(x,\omega)
	+k^2\int_{B_d}\Phi_d(x,y)q(y)u_q(y,\omega)\,dy,
\end{equation}
where
\begin{equation*}
	\Phi_2(x,y)=\frac{i}{4}H_0^{(1)}(k|x-y|),
	\qquad
	\Phi_3(x,y)=\frac{e^{ik|x-y|}}{4\pi |x-y|}.
\end{equation*}
The far-field pattern is then evaluated by
\begin{equation}\label{eq:num_farfield}
	u_q^\infty(\widehat x,\omega)
	=C_d k^2\int_{B_d}e^{-ik\widehat x\cdot y}q(y)u_q(y,\omega)\,dy,
\end{equation}
with
\begin{equation*}
	C_2=\frac{e^{i\pi/4}}{\sqrt{8\pi k}},
	\qquad
	C_3=\frac{1}{4\pi}.
\end{equation*}

The integral equation is discretized on a Cartesian grid in $[-1,1]^d$, after retaining the points lying inside $B_d$.  If $\{x_i\}_{i=1}^{N}$ and $\{w_i\}_{i=1}^{N}$ are the retained nodes and weights, then all incident fields are solved simultaneously from
\begin{equation*}
	\bigl(I-k^2\Phi\operatorname{diag}(q\,w)\bigr)U=U^{\mathrm{inc}}.
\end{equation*}
The far-field matrix is computed from the corresponding Nystr\"om discretization of \eqref{eq:num_farfield}.  In all four experiments the dense far-field computation uses $64$ incident directions and $64$ observation directions.  The two-dimensional experiments use $k=6$, while the three-dimensional experiments use $k=5.5$.

The reported errors are computed on independently generated test data using a different spatial discretization, thereby mitigating the same-discretization inverse crime.  In two dimensions, 4000 development samples are generated on the $N_{\rm side}=25$ disk grid and split into 3200 training and 800 validation samples. The independent test set contains 800 samples generated on the finer $N_{\rm side}=33$ grid. Both development and test data use the same self-cell correction for the weakly singular diagonal entries.  In three dimensions, 5000 physical contrasts are each evaluated on two forward grids, $N_{\mathrm{side}}=15$ and $N_{\mathrm{side}}=16$, yielding 10000 development input--output pairs. The split is performed by physical-sample identity, resulting in 8000 training pairs and 2000 validation pairs. The use of two training discretizations is intended to reduce dependence on a particular forward grid. The independent test set contains 2000 separately generated contrasts evaluated on the unseen $N_{\mathrm{side}}=17$ grid.

\subsubsection{Priors and learning architectures}

We now describe the four learning pipelines.  The common architecture follows the encoder--DNN--reconstructor decomposition discussed earlier in the paper, but the concrete realization of each component depends on whether the prior is infinite-dimensional or finite-dimensional.  Tables~\ref{tab:pipeline_2d} and~\ref{tab:pipeline_3d} summarize these choices for the two- and three-dimensional experiments, respectively.

\paragraph{Infinite-dimensional linked-GRF priors.}
For the infinite-dimensional experiments, the contrast is generated from a latent Gaussian random field $F$ through a positive bounded link and a compactly supported $C^1$ radial cutoff.  We set
\begin{equation}\label{eq:num_linked_grf_prior}
	q(x)=\rho(x)\bigl(\Gamma(F(x))-1\bigr).
\end{equation}
Here $\Gamma$ is the link function
\begin{equation}\label{eq:num_grf_link}
	\Gamma(t)=n_{\min}+(n_{\max}-n_{\min})\frac{1}{1+e^{-t}},
\end{equation}
which maps the latent field into the positive interval $(n_{\min},n_{\max})$.  The cutoff $\rho$ localizes the contrast away from the boundary.  Here $r_{\rm in}=\alpha_\rho R$ and
$r_{\rm out}=\min\{R,r_{\rm in}+\beta_\rho R\}$, with fixed
cutoff parameters chosen separately in two and three dimensions. Writing $r=|x|$, we use
\begin{equation}\label{eq:num_radial_cutoff}
	\rho(x)=\rho(r)=
	\begin{cases}
		1, & r\le r_{\rm in},\\[1mm]
		\dfrac12\left(1+\cos\!\left(\pi\dfrac{r-r_{\rm in}}{r_{\rm out}-r_{\rm in}}\right)\right),
		& r_{\rm in}<r<r_{\rm out},\\[3mm]
		0, & r\ge r_{\rm out}.
	\end{cases}
\end{equation}
Thus
\begin{equation*}
	1+q(x)=(1-\rho(x))+\rho(x)\Gamma(F(x))>0,
\end{equation*}
so the physical positivity constraint is imposed at the level of the prior.  The latent field has mean $0.10$ and a squared-exponential covariance with standard deviation $\sigma_F$ and correlation length $\ell$. In two dimensions, we use $(\sigma_F,\ell,n_{\min},n_{\max})=(0.15,0.22,0.7,1.4)$ and $(r_{\rm in},r_{\rm out})=(0.92R,0.98R)$, while in three dimensions we use $(\sigma_F,\ell,n_{\min},n_{\max})=(0.12,0.28,0.75,1.30)$ and $(r_{\rm in},r_{\rm out})=(0.85R,0.95R)$.

\paragraph{Finite-dimensional Wendland priors.}
For the finite-dimensional experiments, the unknown contrast is generated in a prescribed space
\begin{equation}\label{eq:num_finite_prior}
	W_r=\operatorname{span}\{\psi_1,\ldots,\psi_r\},
	\qquad
	q(x)=\sum_{j=1}^{r}a_j\psi_j(x).
\end{equation}
We use nonnegative compactly supported Wendland $C^2$ radial basis functions together with a radial cutoff.  More precisely, let
\begin{equation}\label{eq:num_wendland_profile}
	\varphi(t)=(1-t)_+^4(4t+1),
\end{equation}
and let $\rho$ be the same type of $C^1$ cutoff as in \eqref{eq:num_radial_cutoff}.  Before weighted normalization, the basis functions are given by
\begin{equation}\label{eq:num_wendland_basis}
	\psi_j^{\rm raw}(x)
	=
	\varphi\!\left(\frac{|x-c_j|}{\rho_b}\right)\rho(x),
	\qquad j=1,\ldots,r,
\end{equation}
where $c_j$ are prescribed centers and $\rho_b>0$ is the basis radius.  The final basis functions $\psi_j$ are obtained from $\psi_j^{\rm raw}$ by columnwise normalization in the weighted discrete $L^2$ norm.  Since both $\varphi$ and $\rho$ are nonnegative, the normalized basis functions remain nonnegative.

In the two-dimensional finite-dimensional experiment, we use $r=17$ basis functions, with centers placed at the origin and on two concentric rings inside the disk.  In the three-dimensional finite-dimensional experiment, we use $r=47$ basis functions, with centers placed at the origin and on three spherical shells inside the ball.  The coefficients are sampled independently from a uniform distribution on $[0,0.025]$ in two dimensions and on $[0,0.01]$ in three dimensions. Hence $a_j\ge 0$, $\psi_j(x)\ge 0$, and the generated contrasts satisfy $1+q(x)>0$.

\begin{table}[H]
\centering
\caption{Two-dimensional learning setup.}
\label{tab:pipeline_2d}
\begin{tabular}{@{}p{0.20\textwidth}p{0.36\textwidth}p{0.36\textwidth}@{}}
\toprule
Component & Infinite-dimensional model & Finite-dimensional model \\
\midrule
Prior
& Linked-GRF prior \eqref{eq:num_linked_grf_prior} on the disk.
& Wendland prior \eqref{eq:num_finite_prior} with $r=17$ basis functions. \\
\addlinespace
Encoder / input
& Fourier encoder on $\SSS^1\times\SSS^1$ with modes $|m|,|n|\le20$, yielding $3362$ real features from the $64\times64$ far-field data.
& Direct measurement encoder based on a complex $16\times16$ incident--observation submatrix, yielding $512$ real features. \\
\addlinespace
DNN
& ReLU MLP with width $256$ and $4$ affine layers; centered input and standardized output coefficients.
& ReLU MLP with width $256$ and $4$ affine layers; componentwise standardization of inputs and coefficients. \\
\addlinespace
Output
& $370$ coefficients in the disk Fourier--Jacobi trunk space.
& Finite coefficient vector $a\in\RR^{17}$. \\
\addlinespace
Reconstruction
& Linear reconstruction in the weighted-QR-orthonormalized disk trunk basis.
& Direct basis reconstruction $\widehat q=\sum_{j=1}^r\widehat a_j\psi_j$. \\
\bottomrule
\end{tabular}
\end{table}

\begin{table}[H]
\centering
\caption{Three-dimensional learning setup.}
\label{tab:pipeline_3d}
\begin{tabular}{@{}p{0.20\textwidth}p{0.36\textwidth}p{0.36\textwidth}@{}}
\toprule
Component & Infinite-dimensional model & Finite-dimensional model \\
\midrule
Prior
& Linked-GRF prior \eqref{eq:num_linked_grf_prior} on the ball.
& Wendland prior \eqref{eq:num_finite_prior} with $r=47$ basis functions. \\
\addlinespace
Encoder / input
& Spherical-harmonic coefficient encoder on $\SSS^2\times\SSS^2$ with degree $L=6$, yielding $4802$ real features from the $64\times64$ far-field data.
& Direct measurement encoder based on a complex $49\times49$ incident--observation submatrix, yielding $4802$ real features. \\
\addlinespace
DNN
& ReLU MLP with width $512$ and $5$ affine layers; centered input and standardized coefficients.
& ReLU MLP with width $512$ and $5$ affine layers; componentwise standardization of inputs and coefficients. \\
\addlinespace
Output
& Coefficients in the ball spherical-harmonic--Jacobi trunk space with total degree at most $14$.
& Finite coefficient vector $a\in\RR^{47}$. \\
\addlinespace
Reconstruction
& Linear reconstruction in the weighted-QR-orthonormalized ball trunk basis.
& Direct basis reconstruction $\widehat q=\sum_{j=1}^r\widehat a_j\psi_j$. \\
\bottomrule
\end{tabular}
\end{table}

All four networks are trained for 180 epochs using AdamW with learning rate $5\times10^{-4}$, weight decay $10^{-6}$, and batch size $64$. Tables~\ref{tab:pipeline_2d} and~\ref{tab:pipeline_3d} summarize the four numerical pipelines.  They are intended to clarify the experimental design rather than to support a direct comparison between the two- and three-dimensional error levels.

\subsubsection{Noiseless reconstructions}

We first evaluate the learned inverse maps on independent test data without additional measurement noise.  Representative reconstructions are shown in Figures~\ref{fig:noiseless_2d} and~\ref{fig:noiseless_3d}, using one randomly selected test example from each experiment.  The two-dimensional figure displays full-disk reconstructions for the infinite- and finite-dimensional priors.  The three-dimensional figure displays central coordinate slices in the $xy$, $xz$, and $yz$ planes.  In two dimensions, the mean relative weighted $L^2$ error is $8.15\%$ for the infinite-dimensional model and $1.32\%$ for the finite-dimensional model.  In three dimensions, the corresponding errors are $8.01\%$ and $6.11\%$. These noiseless errors provide reference levels for the subsequent noise-sensitivity study. The representative reconstructions below illustrate the spatial distribution of the reconstruction errors, whereas the reported percentages are averages over the full independent test sets.

For the finite-dimensional tests, the reference contrast is evaluated on the reconstruction grid from the saved true coefficient vector using the corresponding saved basis normalization. This avoids adding an interpolation error or a basis-normalization mismatch to the reported inverse error.  For the infinite-dimensional tests, the independent-grid reference field is interpolated to the reconstruction grid before the weighted error is computed.

\begin{figure}[!htbp]
	\centering
	\begin{minipage}{0.95\textwidth}
		\centering
		\textbf{(a) Infinite-dimensional prior.}\par\vspace{0.3em}
		\includegraphics[width=\textwidth]{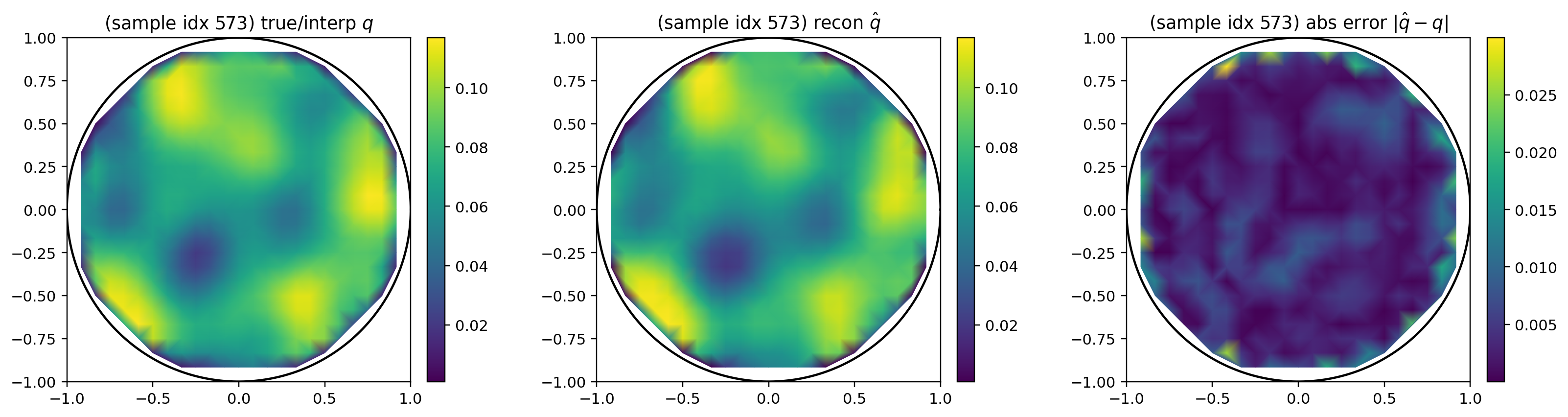}
	\end{minipage}
	
	\vspace{0.8em}
	
	\begin{minipage}{0.95\textwidth}
		\centering
		\textbf{(b) Finite-dimensional prior.}\par\vspace{0.3em}
		\includegraphics[width=\textwidth]{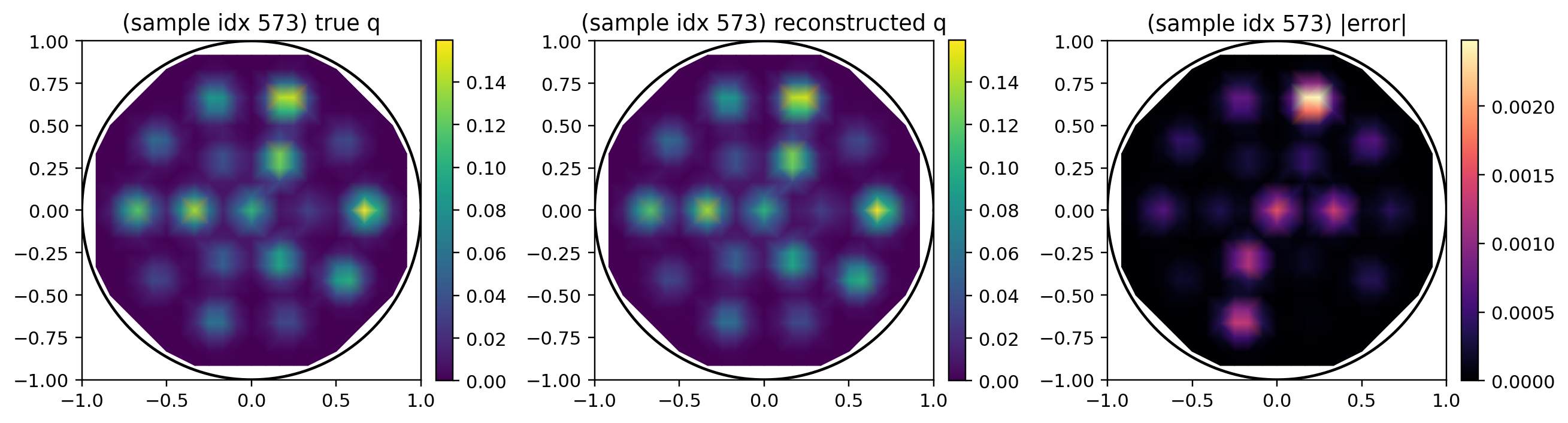}
	\end{minipage}
	\caption{Representative noiseless reconstructions for the two-dimensional experiments, using one randomly selected test example from each setting.  Each image shows the reference contrast, the reconstructed contrast, and the absolute pointwise error.}
	\label{fig:noiseless_2d}
\end{figure}

Figures~\ref{fig:noiseless_2d} and~\ref{fig:noiseless_3d} show that the dominant spatial structures of the contrasts are recovered in all four settings. The infinite-dimensional reconstructions preserve the large-scale features but exhibit more spatially distributed residual errors, whereas the finite-dimensional reconstructions recover the prescribed basis components more accurately. Within each dimension, the smaller error of the finite-dimensional model is consistent with its lower-dimensional output representation and direct prediction of the basis coefficients. Since the two prior classes define different reconstruction tasks, however, these error levels should not be interpreted as a controlled comparison between the two models.

\begin{figure}[p]
	\centering
	\begin{minipage}{0.99\textwidth}
		\centering
		\textbf{(a) Infinite-dimensional prior.}\par\vspace{0.3em}
		\includegraphics[height=0.43\textheight,keepaspectratio]{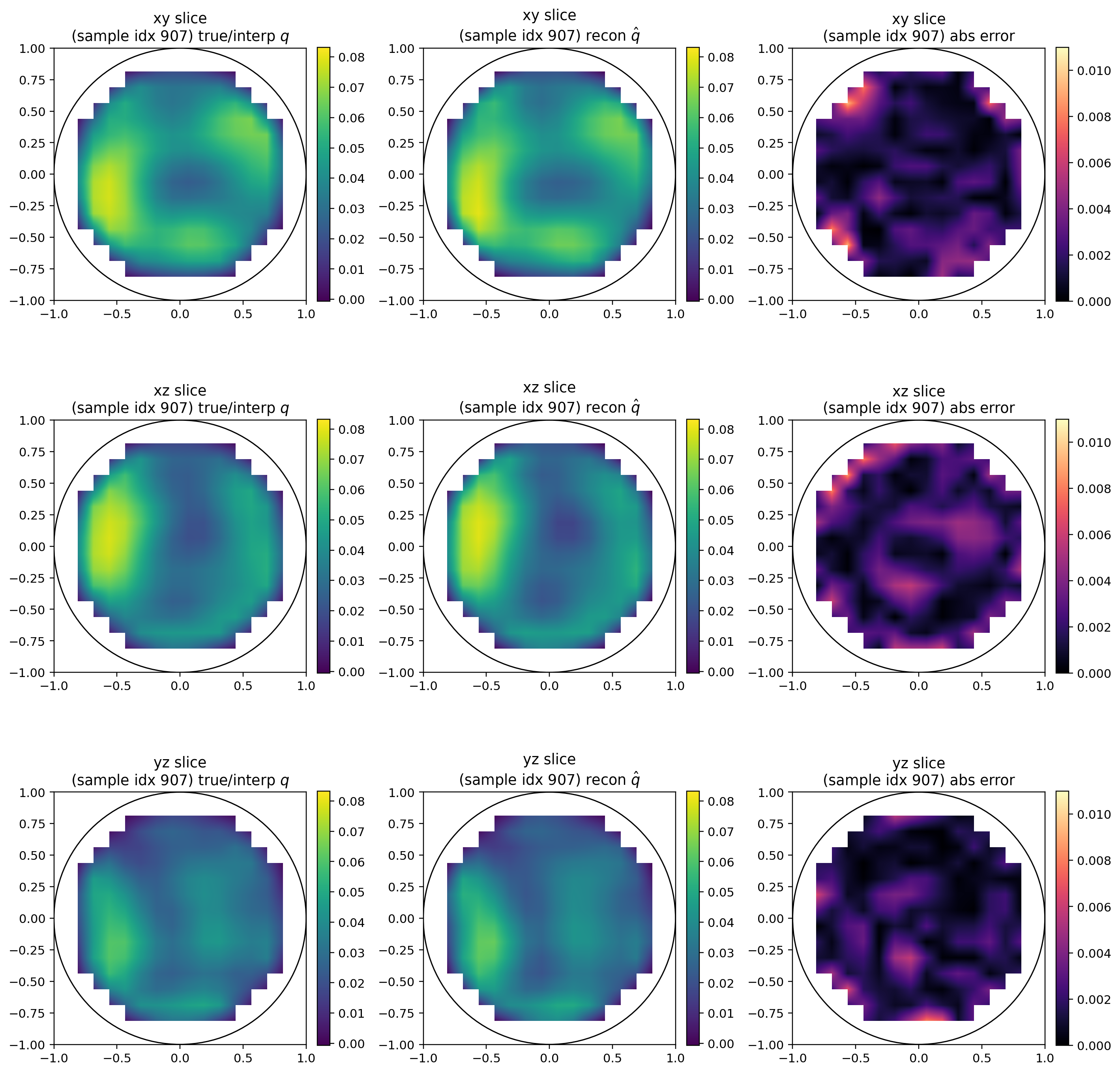}
	\end{minipage}
	
	\vspace{0.6em}
	
	\begin{minipage}{0.99\textwidth}
		\centering
		\textbf{(b) Finite-dimensional prior.}\par\vspace{0.3em}
		\includegraphics[height=0.43\textheight,keepaspectratio]{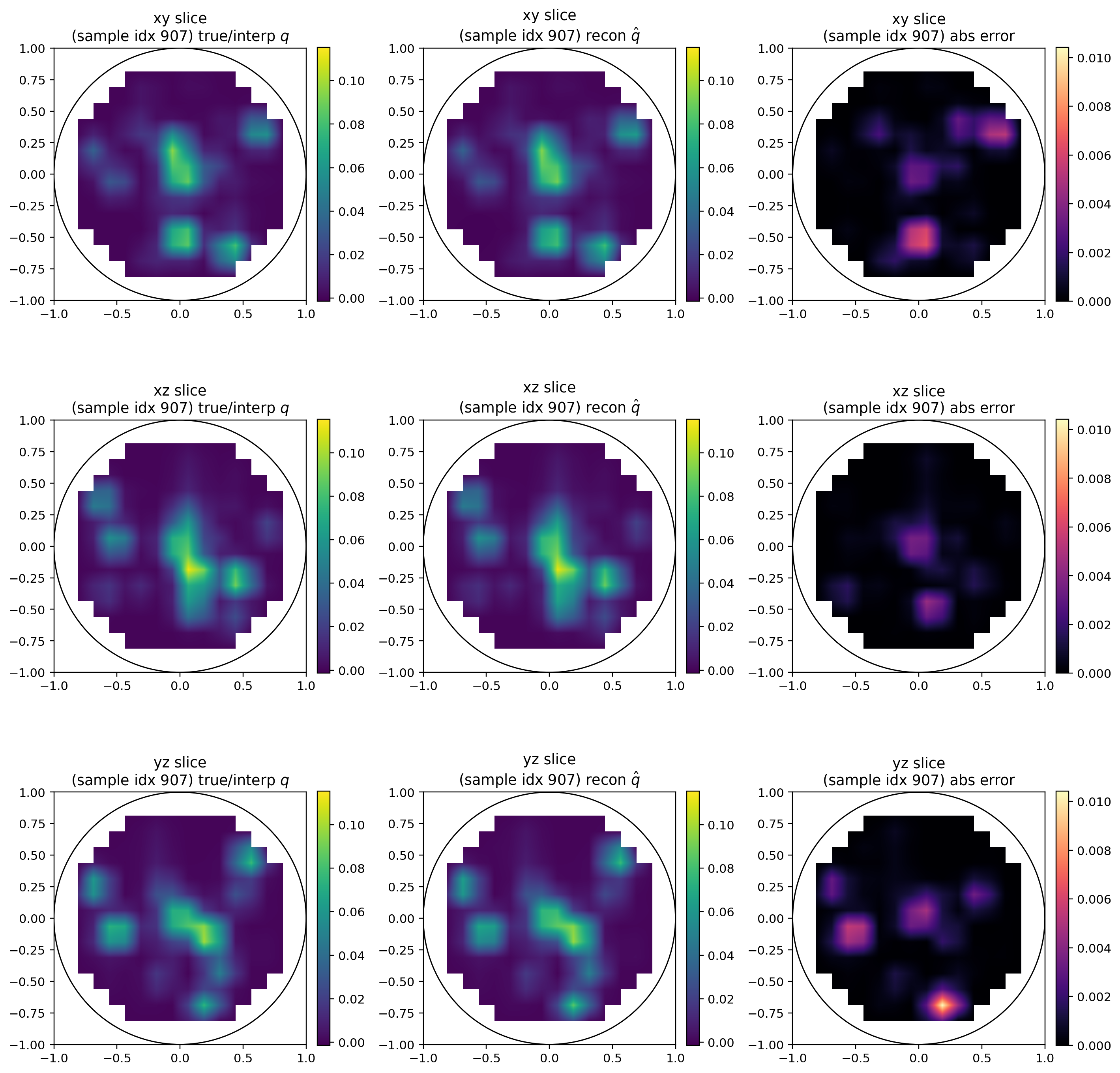}
	\end{minipage}
	\caption{Representative noiseless reconstructions for the three-dimensional experiments, using one randomly selected test example from each setting.  The columns show central $xy$, $xz$, and $yz$ slices, and the rows show the reference contrast, the reconstruction, and the absolute error.}
	\label{fig:noiseless_3d}
\end{figure}

\FloatBarrier

\subsubsection{Sensitivity to measurement noise}

Finally, we examine the sensitivity of the learned inverse maps to
perturbations of the measured far-field data.  We use entrywise multiplicative complex noise,
\begin{equation}\label{eq:num_noise_model}
	(U_\eta)_{ij}
	=
	U_{ij}(1+\eta Z_{ij}),
	\qquad
	Z_{ij}
	=
	\frac{\xi_{ij}^{(1)}+i\xi_{ij}^{(2)}}{\sqrt{2}},
	\qquad
	\xi_{ij}^{(1)},\xi_{ij}^{(2)}
	\stackrel{\mathrm{i.i.d.}}{\sim}\mathcal N(0,1),
\end{equation}
where
\begin{equation*}
	\eta\in\{0,0.005,0.01,0.05\}.
\end{equation*}
For the infinite-dimensional models, the perturbation is applied to the dense far-field matrix before the Fourier or spherical-harmonic encoder.  For the finite-dimensional models, the perturbation is applied to the sampled measurement matrix before real--imaginary vectorization.  The errors reported in Tables~\ref{tab:noise_2d_compact} and~\ref{tab:noise_3d_compact} are
\begin{equation*}
	E_{\rm rel}
	=
	\frac{\left(\sum_i |\widehat q(x_i)-q(x_i)|^2w_i\right)^{1/2}}
	{\left(\sum_i |q(x_i)|^2w_i\right)^{1/2}},
	\qquad
	E_{\rm abs}
	=
	\left(\sum_i |\widehat q(x_i)-q(x_i)|^2w_i\right)^{1/2}.
\end{equation*}
We report the test-set means of these quantities to summarize their variation with the prescribed noise level.

\begin{table}[htbp]
\centering
\caption{Sensitivity to measurement noise in two dimensions.}
\label{tab:noise_2d_compact}
\begin{tabular}{@{}ccccc@{}}
\toprule
Noise level $\eta$
& \multicolumn{2}{c}{Infinite-dimensional}
& \multicolumn{2}{c}{Finite-dimensional} \\
\cmidrule(lr){2-3}\cmidrule(l){4-5}
& $E_{\rm rel}$ & $E_{\rm abs}$ & $E_{\rm rel}$ & $E_{\rm abs}$ \\
\midrule
$0$     & $0.08145$ & $9.580\!\times\!10^{-3}$ & $0.01320$ & $7.735\!\times\!10^{-4}$ \\
$0.005$ & $0.08183$ & $9.625\!\times\!10^{-3}$ & $0.01325$ & $7.765\!\times\!10^{-4}$ \\
$0.01$  & $0.08300$ & $9.765\!\times\!10^{-3}$ & $0.01339$ & $7.853\!\times\!10^{-4}$ \\
$0.05$  & $0.11240$ & $1.326\!\times\!10^{-2}$ & $0.01813$ & $1.068\!\times\!10^{-3}$ \\
\bottomrule
\end{tabular}
\end{table}

\begin{table}[htbp]
\centering
\caption{Sensitivity to measurement noise in three dimensions.}
\label{tab:noise_3d_compact}
\begin{tabular}{@{}ccccc@{}}
\toprule
Noise level $\eta$
& \multicolumn{2}{c}{Infinite-dimensional}
& \multicolumn{2}{c}{Finite-dimensional} \\
\cmidrule(lr){2-3}\cmidrule(l){4-5}
& $E_{\rm rel}$ & $E_{\rm abs}$ & $E_{\rm rel}$ & $E_{\rm abs}$ \\
\midrule
$0$     & $0.08005$ & $5.395\!\times\!10^{-3}$ & $0.06111$ & $2.474\!\times\!10^{-3}$ \\
$0.005$ & $0.08040$ & $5.419\!\times\!10^{-3}$ & $0.06115$ & $2.475\!\times\!10^{-3}$ \\
$0.01$  & $0.08130$ & $5.480\!\times\!10^{-3}$ & $0.06131$ & $2.482\!\times\!10^{-3}$ \\
$0.05$  & $0.10488$ & $7.092\!\times\!10^{-3}$ & $0.06584$ & $2.667\!\times\!10^{-3}$ \\
\bottomrule
\end{tabular}
\end{table}

\begin{figure}[p]
	\centering
	
	\begin{minipage}{0.99\textwidth}
		\centering
		\textbf{(a) Two-dimensional infinite-dimensional experiment.}\par\vspace{0.3em}
		\includegraphics[width=\textwidth]{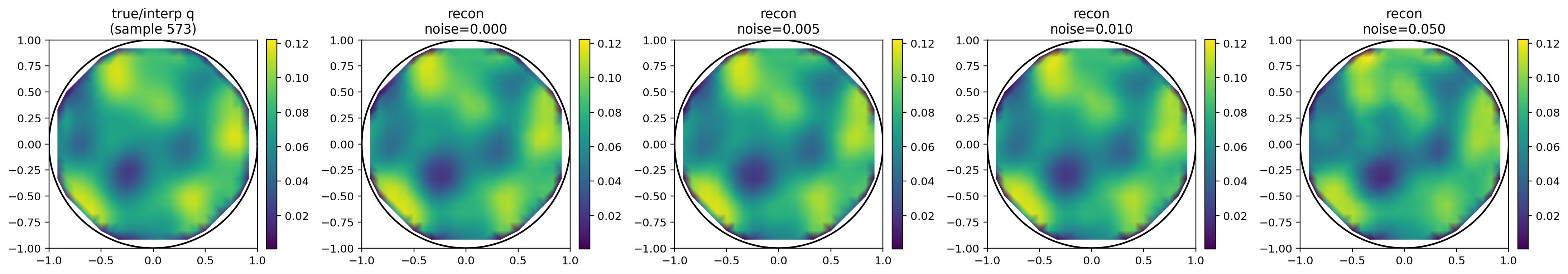}
	\end{minipage}
	
	\vspace{0.8em}
	
	\begin{minipage}{0.99\textwidth}
		\centering
		\textbf{(b) Two-dimensional finite-dimensional experiment.}\par\vspace{0.3em}
		\includegraphics[width=\textwidth]{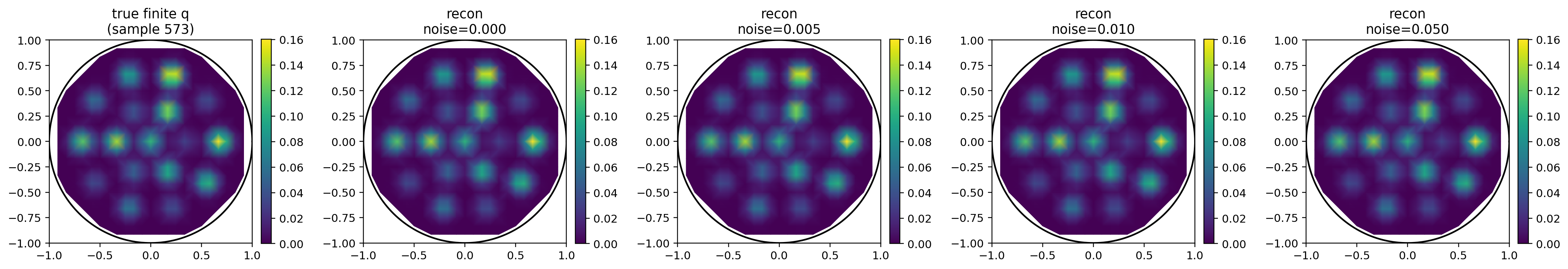}
	\end{minipage}
	
	\vspace{0.8em}
	
	\begin{minipage}{0.99\textwidth}
		\centering
		\textbf{(c) Three-dimensional infinite-dimensional experiment, central $xy$ slice.}\par\vspace{0.3em}
		\includegraphics[width=\textwidth]{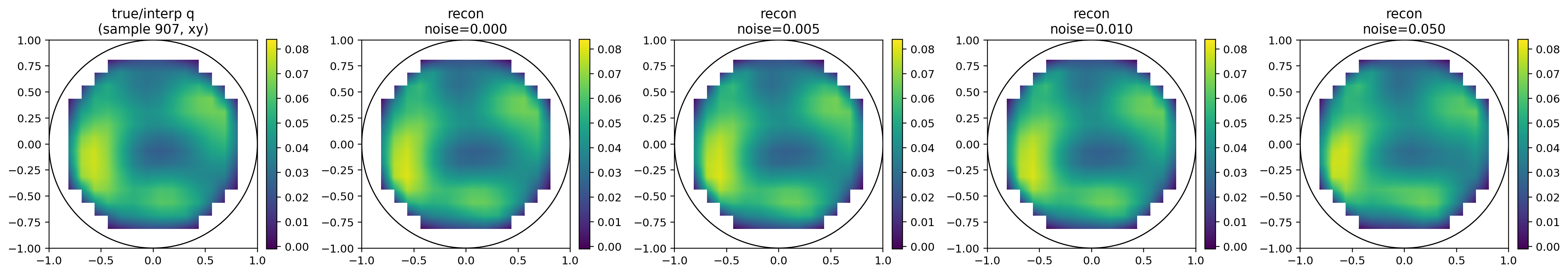}
	\end{minipage}
	
	\vspace{0.8em}
	
	\begin{minipage}{0.99\textwidth}
		\centering
		\textbf{(d) Three-dimensional finite-dimensional experiment, central $xy$ slice.}\par\vspace{0.3em}
		\includegraphics[width=\textwidth]{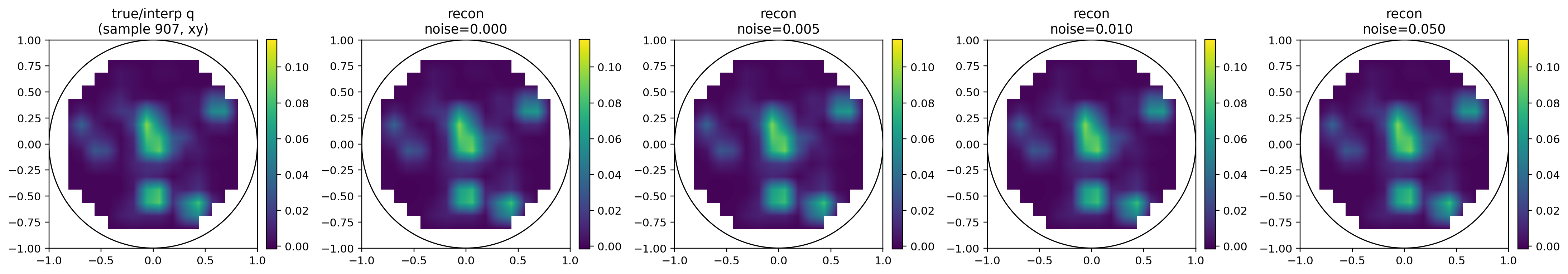}
	\end{minipage}
	
	\caption{Representative reconstructions under measurement noise for the four inverse-scattering experiments. Each panel shows the reference contrast followed by reconstructions obtained from noisy far-field data with $\eta=0,0.005,0.01,0.05$. For the three-dimensional experiments, central $xy$ slices are displayed.}
	\label{fig:noise_stability_visualizations}
\end{figure}

The results in Tables~\ref{tab:noise_2d_compact} and~\ref{tab:noise_3d_compact} show that small far-field perturbations lead to only modest changes in the mean reconstruction errors on the tested distributions.  In two dimensions, the infinite-dimensional error changes only slightly at $0.5\%$ and $1\%$ noise and increases from $8.15\%$ to $11.24\%$ at $5\%$ noise; the finite-dimensional error increases from $1.32\%$ to $1.81\%$.  In three dimensions, the infinite-dimensional error increases from $8.01\%$ to $8.13\%$ at $1\%$ noise and to $10.49\%$ at $5\%$ noise, while the finite-dimensional error increases from $6.11\%$ to $6.58\%$. Figure~\ref{fig:noise_stability_visualizations} presents representative reconstructions under the four prescribed noise levels. The two-dimensional panels display full-disk reconstructions at each noise level, while the three-dimensional panels display the central $xy$ slice across noise levels. 

Overall, the numerical experiments illustrate the reconstruction behavior of the proposed architecture for contrasts drawn from both infinite- and finite-dimensional priors in two and three dimensions. Across all four settings, the mean reconstruction errors remain close to their noiseless values at noise levels of $0.5\%$ and $1\%$, whereas a more noticeable increase is observed at $5\%$. These results provide an empirical assessment of the sensitivity of the trained models to the prescribed measurement noise on the tested distributions.

\section{Conclusion}\label{section5}
We developed an operator-level DeepONet framework for logarithmically stable infinite-dimensional inverse problems. By decomposing the learned inverse map into measurement encoding, finite-dimensional neural approximation, and output reconstruction, we obtained a unified error estimate that separates the contributions of these three components.

For inverse maps satisfying a logarithmic stability estimate, we quantified the finite-dimensional ReLU approximation error in terms of the network size for prescribed encoder and reconstruction ranks. For prescribed encoder and reconstruction dimensions, the derived upper bounds exhibit different dependence on the target accuracy: the logarithmic-stability construction yields an exponential network-size bound, whereas the Lipschitz-stability construction yields an algebraic bound up to a logarithmic factor. This difference is consistent with the stronger ill-posedness of the logarithmically stable inverse problem. The encoding and reconstruction contributions are controlled separately by their respective rank-dependent approximation errors.

We then specialized the abstract framework to three-dimensional fixed-frequency inverse acoustic medium scattering and obtained a quantitative DeepONet error bound for recovering the medium contrast from far-field measurements. Numerical experiments in two and three dimensions, using both infinite- and finite-dimensional prior classes, illustrate the reconstruction performance of the proposed architecture and its empirical sensitivity to synthetic measurement noise. The quantitative inverse-scattering analysis established in this work is three-dimensional, whereas the two-dimensional experiments provide numerical evidence for the same computational framework. Future work includes deriving joint complexity estimates for the encoding, neural approximation, and reconstruction levels, and extending the framework to broader classes of inverse problems with non-Lipschitz stability moduli.

\section*{Acknowledgements}
The authors sincerely thank Professor Jenn-Nan Wang of National Taiwan University for his valuable guidance on this work. T. Li was partially supported by the National Natural Science Foundation of China (NSFC) 12371377 and the Jiangsu Provincial Scientific Research Center of Applied Mathematics under Grant No. BK20233002. This research was partially funded by Shanghai Institute for Mathematics and Interdisciplinary Sciences under grant number SIMIS-ID-2024-LG. We thank Tianhe-2 and the Big Data Computing Center in Southeast University, China, for the use of their computing resources.

\appendix

	\section{Borel measurability of the inverse-scattering map}
	This appendix establishes the measurability of the inverse-scattering map used in the qualitative DeepONet formulation.
	\begin{lemma}[Lusin-Souslin theorem]
		Let $X$ and $Y$ be Polish spaces. If $f: X \to Y$ is continuous and injective, then $f(X)$ is a Borel subset of $Y$. Moreover, the inverse mapping $f^{-1}: f(X) \to X$ is Borel measurable.
	\end{lemma}
	
	\begin{theorem}\label{borel-measurability}
		Let $\mathcal Q\subset C(B)$ be a closed admissible class of contrasts such that the forward scattering map
		\[
			\Psi:\mathcal Q\to C(\mathbb S^{d-1}\times\mathbb S^{d-1}),
			\qquad \Psi(q)=u_q^\infty,
		\]
		is continuous and injective. Then there exists a Borel measurable map
		\[
			\mathscr{G}:L^2(\mathbb{S}^{d-1}\times\mathbb{S}^{d-1})\to L^2(B)
		\]
		such that $\mathscr{G}(u_q^\infty)=q$ for every $q\in\mathcal Q$.
	\end{theorem}
	
	\begin{proof}
		Set $Z:=C(\mathbb{S}^{d-1}\times\mathbb{S}^{d-1})$ and $X:=L^2(\mathbb{S}^{d-1}\times\mathbb{S}^{d-1})$, and let $j:Z\to X$ be the canonical inclusion. Consider the forward scattering operator $\Psi:\mathcal Q\to Z$, $\Psi(q)=u_q^\infty$. By the assumptions of the theorem,
		$\Psi:\mathcal Q\to Z$ is continuous and injective. Hence
		$\widetilde\Psi:=j\circ\Psi:\mathcal Q\to X$ is also continuous and injective.
		
		Since $\mathcal Q$ is closed in $C(B)$, it is a Polish space. Since $X$ is also Polish, the Lusin--Souslin theorem shows that $D:=\widetilde\Psi(\mathcal Q)$ is Borel in $X$ and that $\widetilde\Psi^{-1}:D\to\mathcal Q$ is Borel measurable. Let $i:\mathcal Q\to L^2(B)$ be the canonical inclusion and define
		\[
			\mathscr{G}(f):=
			\begin{cases}
				i\bigl(\widetilde\Psi^{-1}(f)\bigr), & f\in D,\\
				0, & f\in X\setminus D.
			\end{cases}
		\]
		Because $D$ is Borel and $i\circ\widetilde\Psi^{-1}$ is Borel measurable, so is $\mathscr{G}$. By construction, $\mathscr{G}(u_q^\infty)=q$ for every $q\in\mathcal Q$.
	\end{proof}
	
	\section{Proof of the qualitative approximation theorem}
	\label{app:qualitative-approximation}
	We collect here the auxiliary results used in the proof of Theorem~\ref{thm:deeponet-inverse-scattering-approx}.
	\begin{lemma}[Lusin's theorem]
		Let $ X, Y $ be separable and complete metric spaces. Let $\mu \in \mathcal{P}(X)$ be a probability measure, and let $\mathscr{G} : X \to Y$ be a Borel measurable mapping. Then for any $\varepsilon > 0$, there exists a compact set $K \subset X$, such that $\mu(X \setminus K) < \varepsilon$, and such that the restriction $\mathscr{G}|_{K} : K \to Y$ is continuous.
	\end{lemma}
	
	\begin{lemma}[Clipping lemma \cite{LMK22}]
		Let $\varepsilon > 0$, and fix $0 < R_1 < R_2$. There exists a ReLU neural network $\gamma : \mathbb{R}^p \to \mathbb{R}^p$, such that
		\begin{equation*}
			\begin{cases}
				\|\gamma(x) - x\|_{\ell^2} < \varepsilon, & \text{if } \|x\|_{\ell^2} \leq R_1, \\
				\|\gamma(x)\|_{\ell^2} \leq R_2, & \forall x \in \mathbb{R}^p.
			\end{cases}
		\end{equation*}
	\end{lemma}
	
	\begin{lemma}[Universal approximation theorem for operators \cite{CC95}]
		Suppose that $\sigma$ is a continuous non-polynomial function, $X$ is a real Banach space, $K_1 \subset X$, $K_2 \subset \mathbb{R}^d$ are two compact sets in $X$ and $\mathbb{R}^d$, respectively, $V$ is a compact set in $C(K_1;\mathbb R)$, $\mathscr{G}$ is a nonlinear continuous operator, which maps $V$ into $C(K_2;\mathbb R)$, then for any $\varepsilon > 0$, there are positive integers $n, p, m$, constants $c_i^k, \xi_{ij}^k, \theta_i^k, \zeta_k \in \mathbb{R}$, $w_k \in \mathbb{R}^d$, $x_j \in K_1$, $i = 1, \ldots, n$, $k = 1, \ldots, p$, $j = 1, \ldots, m$, such that
		
		\begin{equation*}
			\left| \mathscr{G}(u)(y) - \sum_{k=1}^{p} \sum_{i=1}^{n} c_i^k \sigma \left( \sum_{j=1}^{m} \xi_{ij}^k u(x_j) + \theta_i^k \right) \sigma(w_k \cdot y + \zeta_k) \right| < \varepsilon
		\end{equation*}
		
		holds for all $u \in V$ and $y \in K_2$.
	\end{lemma}
	
	\begin{proof}[Proof of Theorem~\ref{thm:deeponet-inverse-scattering-approx}]
		For complex-valued inputs and outputs, we apply the preceding real-valued result after identifying the real and imaginary input channels with functions on two disjoint copies of the input domain and treating the real and imaginary output components separately; the resulting real coordinates are concatenated. All bases and projections below are understood in the corresponding realified spaces, while the original notation is retained.
		\begin{enumerate}
			\item \textbf{Clipping $\mathscr{G}$:}  
			For $M > 0$, define the clipped operator $\mathscr{G}_M: X \to Y$ by
			\begin{equation*}
				\mathscr{G}_M(f) = 
				\begin{cases}
					\mathscr{G}(f) & \text{if } \|\mathscr{G}(f)\|_{L^2(B)} \leq M, \\
					M \dfrac{\mathscr{G}(f)}{\|\mathscr{G}(f)\|_{L^2(B)}} & \text{if } \|\mathscr{G}(f)\|_{L^2(B)} > M.
				\end{cases}
			\end{equation*}
			By construction, $\|\mathscr G_M(f)\|_{L^2(B)}\le M$ for every
			$f\in X$. Moreover,
			\begin{equation*}
				\| \mathscr{G} - \mathscr{N} \|_{L^2(\mu)} \leq \| \mathscr{G} - \mathscr{G}_M \|_{L^2(\mu)} + \| \mathscr{G}_M - \mathscr{N} \|_{L^2(\mu)},
			\end{equation*}
			and by the dominated convergence theorem, $\|\mathscr{G} - \mathscr{G}_M\|_{L^2(\mu)} \to 0$ as $M \to \infty$. Choose $M > \varepsilon$ such that
			\begin{equation}
				\| \mathscr{G} - \mathscr{G}_M \|_{L^2(\mu)} < \frac{\varepsilon}{3}. \label{eq:2.1}
			\end{equation}
			
			\item \textbf{Lusin's theorem:}  
			Since $C(\mathbb{S}^{d-1} \times \mathbb{S}^{d-1})$ and $Y$ are Polish spaces (separable complete metric spaces) and $\mathscr{G}_M$ is Borel measurable, by Lusin's theorem, there exists a compact set $K \subset C(\mathbb{S}^{d-1} \times \mathbb{S}^{d-1})$ such that $\mu(C(\mathbb{S}^{d-1} \times \mathbb{S}^{d-1}) \setminus K) < (\varepsilon/(9M))^2$ and the restriction $\mathscr{G}_M|_K: K \to Y$ is continuous.
			
			Since $D\subset C(\mathbb{S}^{d-1} \times \mathbb{S}^{d-1})$ and $\mu(D)=1$, we have $\mu(X \setminus C(\mathbb{S}^{d-1} \times \mathbb{S}^{d-1}))=0$. Therefore, $\mu(X \setminus K)=\mu(C(\mathbb{S}^{d-1} \times \mathbb{S}^{d-1}) \setminus K)<(\varepsilon/(9M))^2.$
			Moreover, $K$ is also compact in $X$.
			
			\item \textbf{Finite-dimensional projection:}  
			Let $\{\phi_1, \phi_2, \dots\}$ be an orthonormal basis of $Y = L^2(B)$ consisting of continuous functions (such a basis exists since $B$ is compact). For $\kappa \in \mathbb{N}$, define the projection $P_\kappa: Y \to C(B)$ by
			\begin{equation*}
				P_\kappa(v) = \sum_{k=1}^\kappa \langle v, \phi_k \rangle \phi_k.
			\end{equation*}
			Then $P_\kappa$ is continuous for any fixed $\kappa$. Let $K' = \mathscr{G}_M(K) \subset Y$, which is compact due to the continuity of $\mathscr{G}_M|_K$ and compactness of $K$. Thus, there exists $\kappa \in \mathbb{N}$ such that
			\begin{equation}
				\max_{v \in K'} \| v - P_\kappa(v) \|_{L^2(B)} < \frac{\varepsilon}{6}. \label{eq:2.2}
			\end{equation}
			Define $\mathcal{F} := P_\kappa \circ \mathscr{G}_M: K \subset C(\mathbb{S}^{d-1} \times \mathbb{S}^{d-1}) \to C(B)$. Then $\mathcal{F}$ is continuous, and by \eqref{eq:2.2},
			\begin{equation}
				\begin{aligned}
					\max_{f \in K} \| \mathscr{G}_M(f) - \mathcal{F}(f) \|_{L^2(B)} 
					&= \max_{f \in K} \| \mathscr{G}_M(f) - P_\kappa \circ \mathscr{G}_M(f) \|_{L^2(B)} \\ 
					&= \max_{v \in K'} \| v - P_\kappa(v) \|_{L^2(B)} < \frac{\varepsilon}{6}.
				\end{aligned} \label{eq:2.3}
			\end{equation}
			
			\item \textbf{Universal approximation on a compact set:}  
			By the universal approximation theorem for continuous operators on compact sets, applied to $\mathcal{F}: K \subset C(\mathbb{S}^{d-1} \times \mathbb{S}^{d-1}) \to C(B)$, there exists an operator network $\widetilde{\mathscr{N}}$ with a single hidden layer in the approximator network and a single hidden layer in the trunk network (and with $\tau_0 \equiv 0$), such that
			\begin{equation}
				\sup_{f \in K} \| \mathcal{F}(f) - \widetilde{\mathscr{N}}(f) \|_{L^2(B)} < \frac{\varepsilon}{12}. \label{eq:2.4}
			\end{equation}
			After deleting zero and linearly dependent trunk modes and combining the associated branch coefficients, we apply Gram--Schmidt to the remaining trunks and the corresponding inverse-transpose change of basis to the branch outputs. Relabeling the remaining number as $p$, we can write
			\begin{equation*}
				\widetilde{\mathscr{N}}(f)(y)
				=
				\sum_{k=1}^p \beta_k(f)\tau_k(y),
			\end{equation*}
			where $\beta(f)=\mathcal A(\mathcal E(f))$ and
			$\{\tau_1,\dots,\tau_p\}\subset L^2(B)$ is an orthonormal trunk family. In particular, we have
			\begin{equation*}
				\|\widetilde{\mathscr{N}}(f)\|_{L^2} = \|\beta(f)\|_{\ell^2}, \quad \forall f \in C(\mathbb{S}^{d-1} \times \mathbb{S}^{d-1}).
			\end{equation*}
			For all $f \in K$,
			\begin{equation*}
				\begin{aligned}
					\|\beta(f)\|_{\ell^2} 
					&\leq \|\mathscr{G}_M(f)\|_{L^2(B)} + \|\mathscr{G}_M(f) - \mathcal{F}(f)\|_{L^2(B)} + \|\mathcal{F}(f) - \widetilde{\mathscr{N}}(f)\|_{L^2(B)} \\ 
					&< M + \frac{\varepsilon}{6} + \frac{\varepsilon}{12} < 2M.
				\end{aligned}
			\end{equation*}
			
			\item \textbf{Clipping the DeepONet:}  
			Let $R_1 = M + \varepsilon/4$ and $R_2 = 2M$. By the clipping lemma, there exists a ReLU neural network $\gamma: \mathbb{R}^p \to \mathbb{R}^p$ such that:
			\begin{itemize}
				\item If $\|x\|_{\ell^2} \leq R_1$, then $\|\gamma(x) - x\|_{\ell^2} < \varepsilon/12$,
				\item For all $x \in \mathbb{R}^p$, $\|\gamma(x)\|_{\ell^2} \leq R_2 = 2M$.
			\end{itemize}
			Define the clipped DeepONet $\mathscr{N} = \mathcal{R} \circ \gamma \circ \mathcal{A} \circ \mathcal{E}: C(\mathbb{S}^{d-1} \times \mathbb{S}^{d-1}) \to L^2(B)$ by
			\begin{equation*}
				\mathscr{N}(f)(y) := \sum_{k=1}^p \gamma_k(\beta(f)) \tau_k(y).
			\end{equation*}
			Then for $f \in K$, since $\|\beta(f)\|_{\ell^2} \leq R_1$, we have
			\begin{equation}
				\max_{f \in K} \| \mathscr{N}(f) - \widetilde{\mathscr{N}}(f) \|_{L^2} = \max_{f \in K} \| \gamma(\beta(f)) - \beta(f) \|_{\ell^2} < \frac{\varepsilon}{12}. \label{eq:2.5}
			\end{equation}
			Combining \eqref{eq:2.4} and \eqref{eq:2.5},
			\begin{equation}
				\begin{aligned}
					\max_{f \in K} \| \mathscr{N}(f) - \mathcal{F}(f) \|_{L^2(B)} 
					&\leq \max_{f \in K} \| \mathscr{N}(f) - \widetilde{\mathscr{N}}(f) \|_{L^2(B)} + \max_{f \in K} \| \widetilde{\mathscr{N}}(f) - \mathcal{F}(f) \|_{L^2(B)} \\ 
					&< \frac{\varepsilon}{12} + \frac{\varepsilon}{12} = \frac{\varepsilon}{6}.
				\end{aligned} \label{eq:2.6}
			\end{equation}

			\item \textbf{Error estimation:}  
			We now estimate the error over $X$:
			\begin{equation*}
				\begin{aligned}
					\| \mathscr{G}_M - \mathscr{N} \|_{L^2(\mu)} 
					&\leq \| \mathscr{G}_M - \mathscr{N} \|_{L^2(\mu;K)} + \| \mathscr{G}_M \|_{L^2(\mu;X \setminus K)} \\ 
					&\quad + \| \mathscr{N} \|_{L^2(\mu;X \setminus K)}.
				\end{aligned}
			\end{equation*}
			From \eqref{eq:2.3} and \eqref{eq:2.6}, on $K$,
			\begin{equation*}
				\begin{aligned}
					\| \mathscr{G}_M - \mathscr{N} \|_{L^2(\mu;K)} 
					&\leq \max_{f \in K} \left\{ \| \mathscr{G}_M(f) - \mathcal{F}(f) \|_{L^2(B)} + \| \mathcal{F}(f) - \mathscr{N}(f) \|_{L^2(B)} \right\} \\ 
					&< \frac{\varepsilon}{6} + \frac{\varepsilon}{6} = \frac{\varepsilon}{3}.
				\end{aligned} 
			\end{equation*}

			Since $\|\mathscr{G}_M(f)\|_{L^2(B)} \leq M$ and $\mu(X \setminus K) < (\varepsilon/(9M))^2$,
			\begin{equation*}
				\| \mathscr{G}_M \|_{L^2(\mu;X \setminus K)} \leq M \cdot \mu(X \setminus K)^{1/2} < M \cdot \frac{\varepsilon}{9M} = \frac{\varepsilon}{9}.
			\end{equation*}
			Since $\|\mathscr{N}(f)\|_{L^2} = \|\gamma(\beta(f))\|_{\ell^2} \leq 2M$ for every $f\in C(\mathbb{S}^{d-1}\times\mathbb{S}^{d-1})$,
			\begin{equation*}
				\| \mathscr{N} \|_{L^2(\mu;X \setminus K)} \leq 2M \cdot \mu(X \setminus K)^{1/2} < 2M \cdot \frac{\varepsilon}{9M} = \frac{2\varepsilon}{9}.
			\end{equation*}
			Thus,
			\begin{equation}
				\| \mathscr{G}_M - \mathscr{N} \|_{L^2(\mu)} < \frac{\varepsilon}{3} + \frac{\varepsilon}{9} + \frac{2\varepsilon}{9} = \frac{2\varepsilon}{3}. \label{eq:2.7}
			\end{equation}
			
			Combining \eqref{eq:2.1} and \eqref{eq:2.7},
			\begin{equation*}
				\| \mathscr{G} - \mathscr{N} \|_{L^2(\mu)} \leq \| \mathscr{G} - \mathscr{G}_M \|_{L^2(\mu)} + \| \mathscr{G}_M - \mathscr{N} \|_{L^2(\mu)} < \frac{\varepsilon}{3} + \frac{2\varepsilon}{3} = \varepsilon.
			\end{equation*}
		\end{enumerate}
		This completes the proof of Theorem \ref{thm:deeponet-inverse-scattering-approx}.
	\end{proof}

\end{document}